\documentclass{amsart}
\usepackage[toc,page]{appendix}
\usepackage[utf8]{inputenc}
\usepackage[english]{babel}
\usepackage{blindtext}
\usepackage{amssymb}
\usepackage{enumitem}
\usepackage{amsthm}
\usepackage{amsmath}
\usepackage{caption}
\usepackage{amsfonts}
\usepackage{indentfirst}
\usepackage{comment}
\usepackage{mathtools}
\usepackage{csquotes}
\usepackage{setspace}
\usepackage{physics}
\usepackage{hyperref}
\usepackage{cleveref}
\usepackage{xcolor}
\usepackage[foot]{amsaddr}

\usepackage[
backend=biber,
style=alphabetic,
]{biblatex}
\usepackage[left=3cm,right=3cm,top=4.2cm,bottom=4.2cm, a4paper]{geometry}

\newtheorem{theorem}{Theorem}
\newtheorem{lemma}[theorem]{Lemma}
\newtheorem{proposition}[theorem]{Proposition}
\newtheorem{corollary}[theorem]{Corollary}
\theoremstyle{definition}

\newtheorem{assumption}[theorem]{Assumption}
\theoremstyle{remark}

\newtheorem*{notation}{Notation}

\title[Kinetic Langevin Dynamics Acceleration]{Appropriate State-Dependent Friction Coefficient Accelerates Kinetic Langevin Dynamics}


\author{Keunwoo Lim} 
\address{Department of Applied Mathematics, University of Washington} 

\author{Molei Tao} 
\address{School of Mathematics, Georgia Institute of Technology} 

\email{kwlim@uw.edu, mtao@gatech.edu}

\begin{document}

\begin{abstract}
We consider the convergence of kinetic Langevin dynamics to its ergodic invariant measure, which is Gibbs distribution. Instead of the standard setup where the friction coefficient is a constant scalar, we investigate position-dependent friction coefficient and the possible accelerated convergence it enables. We show that by choosing this coefficient matrix to be $2\sqrt{\text{Hess}\,V}$, convergence is accelerated in the sense that no constant scalar friction coefficient can lead to faster convergence for a large subset of (nonlinear) strongly-convex potential $V$'s. The speed of convergence is quantified in terms of Chi-square divergence from the target distribution, and proved using a Lyapunov approach, based on viewing sampling as optimization in the infinite dimensional space of probability distributions.
\end{abstract}

\maketitle

\section{Introduction}
\noindent
Consider kinetic Langevin dynamics described by a system of stochastic differential equations
\begin{equation}
\begin{cases}
\dd \mathbf{q}_{t} = \mathbf{p}_{t}\, \dd t \\
\dd\mathbf{p}_{t} = -\nabla V(\mathbf{q}_{t})\,\dd t-\Gamma\mathbf{p}_{t}\,\dd t + \Sigma \,\dd\mathbf{W}_{t}.
\end{cases}
\end{equation}
In a physical sense, we understand $\mathbf{q}$ as position, $\mathbf{p}$ as momentum, $V(\mathbf{q})$ as potential function, $\Gamma$ as friction coefficient, and $\Sigma$ as diffusion coefficient. When $\Gamma$ and $\Sigma$ are positive-definite-matrix-valued functions of only $\mathbf{q}$ and $\Sigma(\mathbf{q})=\sqrt{2\Gamma(\mathbf{q})}, \forall \mathbf{q}$, the system admits Gibbs distribution 
\[
    Z^{-1}\exp(-(V(\mathbf{q})+\|\mathbf{p}\|^2/2)) \dd \mathbf{q} \dd \mathbf{p}
\]
as an invariant distribution. Moreover, under reasonable conditions (see later for examples), Gibbs is the unique ergodic distribution of this diffusion process, and $\mathbf{q}_t, \mathbf{p}_t$ converge to it as $t\rightarrow\infty$. Therefore, this dynamics serves as the foundation of many Markov Chain Monte Carlo samplers which discretize the time to construct sampling algorithms, as in \cite{dalalyan_2020, cheng_2018, 19, ma_2021, zhang_2023}. Consequently, its speed of convergence, which is closely related to the efficiency of derived sampling algorithms, is important (together with discretization error; see e.g., \cite{li2022sqrt} for this aspect). We thus focus on analyzing and improving the quantitative convergence rate of this dynamics (no time discretization in this work).

Our improvement is inspired by a simple question: would appropriate choice of state-dependent $\Gamma$, i.e. $\Gamma(\textbf{q})$, possibly accelerate the convergence? The answer is affirmative. We make a specific choice of $\Gamma$ inspired by physical intuitions and quantitatively demonstrate its advantage. The analysis is based on the dynamical evolution of density in the infinite dimensional space of probability density functions, given by the Fokker-Planck equation.

\subsection{Main Results}
Most existing research mainly considered the dynamics where $\Gamma = \lambda I$ for scalar $\lambda>0$. However, in the case of quadratic, positive definite $V$ (i.e. Gaussian sampling), the choice of constant but matrix-valued $\Gamma=2\sqrt{\text{Hess}\, \, V}$ achieves the fastest convergence among all constant matrix-valued $\Gamma$'s \cite{tao2010temperature}. This is the critically damped case in physics terminology, and it was also conjectured in \cite{tao2010temperature} that when $V$ is no longer quadratic but admitting a positive definite Hessian, the choice $\Gamma(\textbf{q})= 2\sqrt{\text{Hess}\, \, V(\textbf{q})}$ would still lead to accelerated convergence, although the meaning of `accelerated' was not quantified. Specifically, particle in nonlinear strongly convex potential spends most of the time near the bottom of the potential well, and it leads to the intuition on quadratic approximation of the potential. We prove the convergence rate in this case is greater than or equal to the $\Gamma =\lambda I$ case for any $\lambda>0$, for a large subset of strongly convex (possibly high-dimensional) $V$'s. 

Our analysis leverage techniques from the seminal work of \cite{ma_2021}, along the line of viewing sampling as optimization in the infinite dimensional space of probability distributions (e.g., \cite{10}). It was known, for example, that overdamped Langevin dynamics corresponds to density evolution that is a gradient flow under Wasserstein-2 metric, where the potential is the Kullback-Leibler (KL) divergence of the distribution of the current state from Gibbs distribution (e.g., \cite{santambrogio2017euclidean}). This potential, therefore, can serve as a Lyapunov function(al) to establish quantitative convergence of the process under reasonable conditions. Kinetic Langevin dynamics, on the other hand, has an additional momentum variable and irreversibility, and it is not known to have a gradient flow structure. Nevertheless, for constant friction coefficients, it has been shown that one can use KL divergence of the joint distribution from the momentum-version Gibbs (e.g., $Z^{-1}\exp(-\|p\|^2/2-V(q))\,dq\,dp$) plus a correction term to construct a Lyapunov function(al) \cite{ma_2021}, which again will allow a quantitative proof of the convergence of kinetic Langevin.

To analyze the convergence of density when the friction coefficient is position-dependent, we introduce a Lyapunov function(al) similar to that in \cite{ma_2021}, which again sums two terms, one being a statistical divergence of joint distribution from target, and another being a Fisher-information-like cross term to account for the fact that kinetic Langevin's Fokker Planck equation admits no known gradient structure. But we replaced  KL divergence in \cite{ma_2021} by $\chi^{2}$-divergence and also modified the cross term in order to obtain a proof that works for our case, where $\Gamma$ is no longer a constant but dependent on $\bf{q}$.

More technical innovations include: 1) when the friction coefficient is not a constant, additional terms absent in the previous research will appear as the Lyapunov function needs to be differentiated. We explicitly evaluate and bound these terms. 2) We provided a sharper convergence rate bound by tuning coefficients in the Lyapunov function. Specifically, we proved that the exponential convergence holds for a parameterized family of Lyapunov functions, and then optimized the convergence rate as a function of parameters of the Lyapunov function. Based on this approach, we obtained the improved convergence rate in \cref{cor17}.

We note that the explicit evaluation of the derivative of our Lyapunov function in \cref{thm5} is valid for any friction coefficient as long as it is a symmetric-matrix-valued function of position, even without needing to be positive definite. Therefore, our approach can also establish the convergence rate for a more general class of friction coefficients. Our main result is \cref{thm1}, which is based on the result of \cref{thm20} and \cref{thm21}. It presents exponential decay of $\chi^{2}$-divergence of probability density function from stationary distribution. $\chi^{2}$-divergence is defined in \cref{chi}.

\begin{theorem}\label{thm1} Consider $2d$-dimensional stochastic differential equation with the potential function  $V(\mathbf{q})\colon\mathbb{R}^{d}\to\mathbb{R}$ and the friction coefficient $\Gamma(\mathbf{q})\colon\mathbb{R}^{d}\to\mathbb{R}^{d\cross d}$ given by
\begin{equation}\label{eq: original}
\begin{cases}
\dd \mathbf{q}_{t} = \mathbf{p}_{t}\, \dd t \\
\dd\mathbf{p}_{t} = -\nabla V(\mathbf{q}_{t})\,\dd t-\Gamma(\mathbf{q}_{t})\mathbf{p}_{t}\,\dd t + \sqrt{2\Gamma(\mathbf{q}_{t})}\,\dd\mathbf{W}_{t}
\end{cases}
\end{equation}
Assume that $\alpha I\preceq \text{Hess}\, V\preceq \beta I$ for $\alpha, \beta>0$ and $\Vert \frac{\partial\sqrt{\text{Hess}\, \, V}}{\partial q_{i}}\Vert_{2} \leq \gamma $ for all $1\leq i\leq d$. Denote the probability density function at time $t$ by $\rho_{t}$. The stationary measure is Gibbs distribution, whose density is given by 
\[
    Z^{-1}\exp(-(V(\mathbf{q})+\|\mathbf{p}\|^2/2)),
\]
denoted by $\pi$. 

When the friction coefficient is given by $\Gamma =2\sqrt{\text{Hess}\, \, V}$,  
\begin{align}
\chi^{2}(\rho_{t}\Vert \pi) \leq C_{1} e^{-\sqrt{\alpha}(\frac{1}{2}-\gamma^{2}m)t}    
\end{align}
holds for some $C_{1} = C_{1}(\rho_{0})$, $C_{2} = C_{2}(\rho_{0})$, $m = m(\alpha, \beta, d)$. Denoting the condition number as $\kappa = \frac{\beta}{\alpha}$, convergence rate is $\frac{\sqrt{\alpha}}{2}-O(\alpha^{-\frac{1}{2}}\kappa^{2}\gamma^{2}d)$.

In particular, considering the case of diagonal quadratic potential $V(\mathbf{q}) = \frac{\alpha}{2}\Vert\mathbf{q}\Vert_{2}^{2}$, for any $\epsilon>0$
\begin{align}
\chi^{2}(\rho_{t}\Vert \pi) \leq C e^{-\sqrt{\alpha}(2-\epsilon)t}    
\end{align}
holds for some $C = C(\rho_{0}, \epsilon)$ independent of $t$. 
\end{theorem}
In \cref{section2}, we introduce basic notations and assumptions, then rewrite the dynamics in a canonical form via appropriate rescaling. In \cref{sec3}, we analyze dynamics in canonical form with friction coefficient $\Gamma = 2\sqrt{\text{Hess}\, \, V}$ and obtain the exponential convergence of Lyapunov function. Lastly, in \cref{app1}, we compare our result with the result of constant scalar friction coefficient case, then analyze the convergence of a dynamics with diagonal quadratic potential function for the support of a claim. 

\subsection{Related Work}
\emph{(in this description, `dynamics' will be in continuous time and `algorithm' will be in discrete time)}
Non-asymptotic sampling error bounds of kinetic Langevin Monte Carlo algorithms (KLMC) have been established in \cite{cheng_2018, dalalyan_2020, zhang_2023, ma_2021} and dimension and accuracy dependence have been quantified. In \cite{26}, composition of synchronous and reflection coupling method is applied to prove the convergence of underdamped Langevin dynamics (ULD). In \cite{8}, improved $L^{2}$-distance convergence for ULD is obtained with functional inequalities. Variant of ULD and KLMC based on numerical approach are introduced in \cite{28, 19}, and further references include \cite{li2022sqrt, 24, 33, 32}. We also note the works on overdamped Langevin dynamics (OLD) and Langevin Monte Carlo (LMC) methods. Starting from the research of \cite{31, durmus2016sampling, durmus2017nonasymptotic}, non-asymptotic error bounds of LMC methods continue to be analyzed \cite{34, 25, 14, durmus2019analysis}. The convexity assumptions could be relaxed inside a ball \cite{Cheng2020} or relaxed by isoperimetric inequality assumptions \cite{chewi2024book}. 

Lyapunov methods provide powerful approaches to understand the convergence of stochastic processes including Langevin dynamics. The classical approach employs a Lyapunov function defined in the state space, which provides a stochastic quantity that monitors the convergence of the SDE. For example, in \cite{5}, the ergodicity of the solution of stochastic differential equation and its discretizations are examined based on Markov chain theory on general state space, specifically for kinetic Langevin dynamics based on a Lyapunov method. More results in this approach are presented in, e.g., \cite{16, 17, 18}. A later approach, on the other hand, considers Lyapunov functions defined in the space of probability measures; they are deterministic quantities that monitor the convergence of the Fokker-Planck PDE instead. This approach corresponds to the perspective of viewing sampling as optimization in the space of probability distributions. For instance, one such result was presented in \cite{ma_2021}, where Lyapunov analyses based on entropic hypocoercivity \cite{villani2009hypocoercivity} were conducted to quantify the convergences of both kinetic Langevin (the continuous process) and its discretization. We also refer to \cite{feng2021hypoelliptic} for some Lyapunov-based results for kinetic Langevin dynamics, which may not be tight but considered variable friction coefficient in the scalar case without focusing on our choice of $\Gamma$ or its acceleration, to \cite{leimkuhler2020hypocoercivity} which employed clever thermostats techniques to choose friction coefficient based on the introduction of an additional variable, and to \cite{chak2023optimal} for a different setup where beautiful analysis is provided to study friction coefficient optimal in the sense of asymptotic variance, as opposed to the sense of convergence rate considered in this article.

An additional remark is, our Lyapunov analysis of continuous time sampling dynamics is closely related to that of continuous time dynamics for optimization. Some seminal work that connect the two include \cite{14,10,25,31}, although overdamped Langevin was the primary consideration. The acceleration effect of introducing momentum in optimization has been analyzed in \cite{9}. The sampling analogue was considered in \cite{ma_2021}, and related notions can trace back to at least entropic hypocoercivity in \cite{villani2009hypocoercivity}. We also note \cite{chak2023optimal} that aimed to find the optimal constant friction matrix that minimizes asymptotic variances of expectation of observables, which is different from our method that minimizes the $\chi^{2}$-divergence between the density and the stationary distribution.

\subsection{Comparison with previous results} Several previous results related to our research are summarized in \cref{table1}. Here, we assume that $\alpha I\preceq\text{Hess}\, \, V \preceq \beta I$ and denote condition number as $\kappa = \frac{\beta}{\alpha}$. Condition $-\gamma I \preceq\frac{\partial\sqrt{\text{Hess}\, \, V}}{\partial q_{i}} \preceq \gamma I$ is required only for our case to analyze the derivative of friction coefficient. We note that our choice of friction coefficient $\Gamma = 2\sqrt{\text{Hess}\, \, V}$ is unique compared to previous works that considered friction coefficient as identity matrix.

    \begin{table}[tbhp]
    \footnotesize
    \caption{Convergence rate of ULD under the condition $\alpha I\preceq\text{Hess}\, \, V \preceq \beta I$}\label{table1}
    \begin{center}
    \begin{tabular}[t]{|c|c|c|c|} 
    \hline
    Source & Convergence Rate & Friction $\Gamma$ & Metric \\ [0.5ex]
    \hline\hline
    Dalalyan and Riou-Durand \cite{dalalyan_2020} & $O(\sqrt{\frac{\alpha}{\kappa}})$ & $O(\sqrt{\beta}) I$& $W_{2}$ \\
    \hline
    Ma et al. \cite{ma_2021}, after rescaling &$O(\sqrt{\frac{\alpha}{\kappa}})$ & $O(\sqrt{\beta}) I$ & twisted KL \\
    \hline
    Cao et al. \cite{8} & $O(\sqrt{\alpha})$ & $O(\sqrt{\alpha}) I$ & $L^{2}$ \\
    \hline
    This paper &$\frac{\sqrt{\alpha}}{2} - O(\alpha^{-\frac{1}{2}}\kappa^{2}\gamma^{2}d)$& $2\sqrt{\text{Hess}\, \, V}$ & $\chi^{2}$ \\
    \hline
    \end{tabular}
    \end{center}
    \end{table}
First observation is that our result can be interpreted as extension from the results of $\Gamma = O(\sqrt{\beta}) I$ and $\Gamma = O(\sqrt{\alpha})I$ cases. Since the smallest eigenvalue of the friction coefficient $\Gamma = 2\sqrt{\text{Hess}\, \, V}$ is $O(\sqrt{\alpha})$ and largest eigenvalue is $O(\sqrt{\beta})$, we may view that our dynamics have the properties of both $\Gamma = O(\sqrt{\alpha})I$ and $\Gamma = O(\sqrt{\beta}) I$ dynamics. It coincides with the fact that convergence rate of our dynamics is $\frac{\sqrt{\alpha}}{2} - O(\alpha^{-\frac{1}{2}}\kappa^{2}\gamma^{2}d)$, which is between $O(\sqrt{\frac{\alpha}{\kappa}})$ and $O(\sqrt{\alpha})$ under reasonable conditions.

We discuss our result in the perspective of our intuition. When the potential function is diagonal quadratic, $\Gamma = 2\sqrt{\text{Hess}\, \, V}$ case have equal or better convergence result than $\Gamma = \lambda I$ case, and the equality holds for the friction coefficient $(2\sqrt{\alpha}) I$. We compare our result for $\Gamma = 2\sqrt{\text{Hess}\, \, V}$ case with the results of two cases of friction coefficients $O(\sqrt{\alpha}) I$ and $O(\sqrt{\beta}) I$, although the measures are not explicitly comparable, and these matchings are only providing handwavy sanity check. As long as $\alpha^{-1}\kappa^{2}\gamma^{2}d$ is small, the order of our convergence rate matches $O(\sqrt{\alpha})$ which is the best outcome as in \cite{8} that the coefficient is not explicitly determined. Also, our result has better condition number dependence than the results with the friction coefficient $O(\sqrt{\beta}) I$ as considered in \cite{dalalyan_2020} and \cite{ma_2021}.

For more explicit analysis, we compared our result with the recent explicit convergence rate result of constant scalar friction coefficient case. We referred the result in \cite[Theorem 4]{8} obtained by DMS hypocoercive estimation approach originated from \cite{32}. In \cref{thm17}, we directly compared the two convergence rates and proved that our convergence rate is higher than the convergence rate of any constant scalar friction coefficient case when $\alpha^{-1}\kappa^{2}\gamma^{2}d$ is small. 

We applied our approach to the diagonal quadratic potential case also considered in \cite[Theorem 3]{8} and obtained a convergence result in \cref{thm21} which is very close to the optimal rate. Moreover, in \cref{theorem19}, we gave explicit comparison for the dynamics with diagonal quadratic potential function as a special case to support the claim.

Finally, we note that our choice of friction coefficient $\Gamma = 2\sqrt{\text{Hess}\, \, V}$ can be interpreted as nonlinear version of critically damped Langevin dynamics. This is because when the potential $V$ is diagonal quadratic, kinetic Langevin with $\Gamma = 2\sqrt{\text{Hess}\, \, V}$ is critically damped in the classical sense and fastest convergent.
\section{Preliminaries}\label{section2}
\begin{notation}
For the function $v\colon\mathbb{R}^{k}\to\mathbb{R}^{k}$, we define the gradient $\nabla v\in \mathbb{R}^{k\times k}$ and the divergence $\nabla\cdot v\in\mathbb{R}$ as 
\begin{align}
(\nabla v)_{ij} = \frac{\partial v_{i}}{\partial x_{j}}\text{ and }\nabla\cdot v= \sum_{i}\frac{\partial v_{i}}{\partial x_{i}}.    
\end{align}
We denote the gradient operator respect to $\mathbf{x}$ as $\nabla_{\mathbf{x}}$, respect to $\mathbf{q}$ as $\nabla_{\mathbf{q}}$, and respect to $\mathbf{p}$ as $\nabla_{\mathbf{p}}$. Then, for the function $M\colon\mathbb{R}^{k}\to\mathbb{R}^{k\times k}$, we define the divergence $\nabla\cdot M \in \mathbb{R}^{k}$ as \begin{align}
(\nabla\cdot M)_{i} = \sum_{j}\frac{\partial M_{ij}}{\partial x_{j}}.    
\end{align}

Next, we define an adjoint operator of gradient with respect to stationary distribution. Letting the stationary distribution of the process as $\pi(\mathbf{q},\mathbf{p})$ and $\mathbf{x}= (\mathbf{q},\mathbf{p})$, we define the adjoint operator of $\nabla_{\mathbf{x}}$ with respect to inner product $\mathbb{E}_{\pi}[\langle\cdot,\cdot\rangle]$ as $(\nabla_{\mathbf{x}}\cdot)^{*}$ and denote it by $(\nabla_{\mathbf{x}})^{*}$, absuing the notation for brievity. Note that the stationary distribution is a Gibbs measure $\pi(\mathbf{q}, \mathbf{p}) =\frac{1}{Z}\exp(-V(\mathbf{q})-\Vert \mathbf{p} \Vert^{2}/2)$
where $Z$ is a normalizing constant. The operation $(\nabla_{\mathbf{x}})^{*}A$ for $2d\times 2d$ matrix $A$ is defined as $2d$-dimensional vector with entries 
\begin{align}
\big((\nabla_{\mathbf{x}})^{*}A\big)_{i}=\sum_{j}(\frac{\partial}{\partial x_{j}})^{*}A_{ij}.    
\end{align}
The operations of $(\nabla_{\mathbf{q}})^{*}$ and $(\nabla_{\mathbf{p}})^{*}$ are similarly defined for the matrix case.

Then, for a function $h\colon \mathbb{R}^{2d}\to\mathbb{R}$, we define an operator $B[h]$ as
\begin{align}
B[h] =(\nabla_{\mathbf{p}})^{*}\nabla_{\mathbf{q}}h-(\nabla_{\mathbf{q}})^{*}\nabla_{\mathbf{p}}h.    
\end{align}
Since the operator $(\nabla_{\mathbf{x}})^{*}$ is given as
$(\nabla_{\mathbf{x}})^{*} = -\nabla_{\mathbf{x}}^\top - (\nabla_{\mathbf{x}}\log\pi)^\top$, $(\nabla_{\mathbf{q}})^{*} = -\nabla_{\mathbf{q}}^\top+\alpha\nabla V(\mathbf{q})^\top$ and $(\nabla_{\mathbf{p}})^{*} = -\nabla_{\mathbf{p}}^\top+\alpha\mathbf{p}^\top$, $B[h]$ is evaluated as $\alpha\mathbf{p}^{\top}\cdot\nabla_{\mathbf{q}}h-\alpha\nabla V(\mathbf{q})^{\top}\cdot\nabla_{\mathbf{p}}h$. Additionally, $[\nabla_{\mathbf{p}}, B][h] =\alpha\nabla_{\mathbf{q}}h$ and $[\nabla_{\mathbf{q}}, B][h] =-\alpha\text{Hess}\,V\nabla_{\mathbf{p}}h$ holds.
For two operators $G, H$, we define a commutator of two operators as
\begin{align}
[G, H]= GH-HG.    
\end{align}

We use $\chi^{2}$-divergence as the metric of two probability distributions $\mu$ and $\nu$, which is defined as 
\begin{align}\label{chi}
    \chi^{2}(\mu \Vert \nu) = \int\big(\frac{\dd \mu}{\dd \nu}-1\big)^{2}\, \dd \nu.
\end{align}
\end{notation}
Then, based on the notation, we introduce the lemma widely applied in the evaluation.
\begin{lemma}
The friction coefficient $\Gamma(\mathbf{q})$ and the operator $(\nabla_{\mathbf{p}})^{*}$ commute, i.e. for $d\cross d$ matrix $A$, $ (\nabla_{\mathbf{p}})^{*}\Gamma A = \Gamma (\nabla_{\mathbf{p}})^{*} A$ holds.
\end{lemma}
\begin{proof}
By direct computation,
\begin{align}
((\nabla_{\mathbf{p}})^{*}\Gamma A)_{i} &=\sum_{j}\alpha p_{j}(\Gamma A)_{ij}-\frac{\partial(\Gamma A)_{ij}}{\partial p_{j}} \\&=\sum_{k}\Gamma_{ik}(\alpha\sum_{j}p_{j}A_{kj}-\sum_{k}\frac{\partial A_{kj}}{\partial p_{j}}) = (\Gamma(\nabla_{\mathbf{p}})^{*}A)_{i}
\end{align}
holds.
\end{proof}

Then, we state the well known Bakry-Emery criterion for strongly convex function (see \cite[Section 4.5]{pavliotis2014stochastic}). We refer PI as Poincar\'e's inequality for convenience. We note that the Bakry-Emery criterion holds for the potential function $V$, which is strongly convex.
\begin{proposition}[Bakry-Emery Criterion]\label{prop33} When a function $U(x)$ satisfies 
\begin{align}
 \text{Hess}\, U \succeq \nu I
\end{align}
for $\nu>0$, then the probability measure $\mu(dx) = \frac{1}{Z}e^{-U} dx$ with normalizing constant $Z$ satisfies PI with constant $\nu$, that is
\begin{align}
\int(f - \int f d\mu)^{2}\leq \frac{1}{\nu} \int \Vert \nabla f \Vert^{2} \dd \mu
\end{align}
holds for function $f\in L^{2}(\mathbb{R}^{d},\mu) \cap C^{2}(\mathbb{R}^{d})$.
\end{proposition}

We also note the core component of our analysis, which is the Cayley-Hamilton Theorem. 
\begin{proposition}[Cayley-Hamilton Theorem]\label{thm: cayley}
The set of the eigenvalues of the square matrix $A$ is equal to the set of the roots of the charateristic polynomial of $A$. 
\end{proposition}

Now, we rescale the original dynamics \cref{eq: original} to revise the Poincar\'e constant of stationary distribution for effective computation. Since the potential function is $\alpha$-strongly convex, we rescale the variable $\mathbf{p}$ and associated variables to make the potential function of the gibbs measure be $\alpha$-strongly convex. We will analyze the rescaled dynamics \cref{eq: rescaled} and prove the convergence in \cref{sec3}.
\begin{proposition}\label{prop5}
By the proper rescaling of the variables of original dynamics \cref{eq: original}, we could obtain the rescaled dynamics \cref{eq: rescaled} with the stationary distribution having PI constant as $\alpha$.
\begin{equation}\label{eq: rescaled}
\begin{cases}
\dd \mathbf{q}_{t} = \mathbf{p}_{t}\, \dd t \\
\dd\mathbf{p}_{t} = -\nabla V(\mathbf{q}_{t})\,\dd t-\Gamma(\mathbf{q}_{t})\mathbf{p}_{t}\,\dd t + \sqrt{2\alpha^{-1}\Gamma(\mathbf{q}_{t})}\,\dd\mathbf{W}_{t}
\end{cases}
\end{equation}
\end{proposition}
\begin{proof}
In the original dynamics \cref{eq: original}, we rescale the potential function $V$, friction coefficient $\Gamma$ and momentum $\mathbf{p}$ as 
$\tilde{V} = \frac{1}{\alpha}V$, $\tilde{\Gamma} = \frac{1}{\sqrt{\alpha}}\Gamma$, and $\mathbf{p}' = \frac{1}{\sqrt{\alpha}}\mathbf{p}$. Then, the equation is equivalent to 
\begin{equation}
\begin{cases}
\dd \mathbf{q}_{t} =\sqrt{\alpha} \mathbf{p}'_{t}\, \dd t \\
\dd \mathbf{p}'_{t} = -\sqrt{\alpha}\nabla \tilde{V}(\mathbf{q}_{t})\,\dd t-\sqrt{\alpha}\tilde{\Gamma}(\mathbf{q}_{\tau})\mathbf{p}'_{t}\,\dd t + \sqrt{2\sqrt{\alpha}^{-1}\tilde{\Gamma}(\mathbf{q}_{t})}\,\dd\mathbf{W}_{t}.
\end{cases}
\end{equation}
Then, we rescale the time variable by defining the process $\tilde{X}_{\tau} = (\tilde{\mathbf{q}}_{\tau}, \tilde{\mathbf{p}}_{\tau})$ as $\tilde{\mathbf{q}}_{\tau} = \mathbf{q}_{\tau/\sqrt{\alpha}}$ and $\tilde{\mathbf{p}}_{\tau} = \mathbf{p}'_{\tau/\sqrt{\alpha}}$. The process $(\tilde{X}_{\tau})$ satisfies the equation 
\begin{equation}\label{eq2}
\begin{cases}
\dd \tilde{\mathbf{q}}_{\tau} = \tilde{\mathbf{p}}_{\tau}\, \dd \tau \\
\dd\tilde{\mathbf{p}}_{\tau} = -\nabla \tilde{V}(\tilde{\mathbf{q}}_{\tau})\,\dd \tau-\tilde{\Gamma}(\tilde{\mathbf{q}}_{\tau})\tilde{\mathbf{p}}_{\tau}\,\dd \tau + \sqrt{2\alpha^{-1}\tilde{\Gamma}(\tilde{\mathbf{q}}_{\tau})}\,\dd\mathbf{W}_{\tau}.
\end{cases}
\end{equation}
By rescaling, the stationary distribution $\tilde{\pi}$ of \cref{eq2} is proportional to $\exp\{-\alpha \tilde {V}-\frac{\alpha}{2} \Vert \tilde{\mathbf{p}}\Vert^{2}\}$, and it has the PI constant as $\alpha$.
\end{proof}
Next is the corollary of the rescaling scheme for the Lyapunov function that we mainly consider in \cref{sec3}. 
\begin{corollary}\label{cor3}
Recalling the notation on the proof of \cref{prop5}, we let the solution of the original dynamics \cref{eq: original} as $\rho_{t}$, the corresponding solution of the rescaled dynamics \cref{eq: rescaled} as $\tilde{\rho}_{t}$, and similarly denote the variables. Denote the Lyapunov function $\tilde{\mathcal{L}}(\tilde{\rho}_{\tau})$ for the rescaled dynamics as
\begin{align*}
     \tilde{\mathcal{L}}(\tilde{\rho}_{\tau})=\chi^{2}(\tilde{\rho}_{\tau}\Vert\tilde{\pi})+\mathbb{E}_{\tilde{\pi}}\big[\,\big\langle\, \nabla_{\tilde{x}}\frac{\mathbf{\tilde{\rho}}_{\tau}}{\tilde{\pi}},\,\tilde{S}\,\nabla_{\tilde{x}}\frac{\mathbf{\tilde{\rho}}_{\tau}}{\tilde{\pi}}\,\big\rangle\,\big],
\end{align*}
where $\tilde{S}$ is a positive semidefinite matrix
$\tilde{S}=
\bigl[ \begin{smallmatrix}b\tilde{\Gamma}^{-2}\,&a\tilde{\Gamma}^{-1}\,\\a\tilde{\Gamma}^{-1}\,&cI\,
\end{smallmatrix}\bigr]=\bigl[\begin{smallmatrix}
b(\tilde{\Gamma})&a(\tilde{\Gamma})\\
a(\tilde{\Gamma})&c(\tilde{\Gamma})
\end{smallmatrix}\bigr]$ with positive coefficients $a, b$, and $c$. In addition, let the Lyapunov function for the original dynamics as
\begin{align}
\mathcal{L}(\rho_{t}) = \chi^{2}(\rho_{t}\Vert\pi)+\mathbb{E}_{\pi}\big[\,\big\langle\, \nabla_{\mathbf{x}}\frac{\rho_{t}}{\pi},\,P^{\top}\tilde{S}P\,\nabla_{\mathbf{x}}\frac{\rho_{t}}{\pi}\,\big\rangle\,\big]
\end{align}
with the matrix $P = \bigl[ \begin{smallmatrix}
I&O\\O& \sqrt{\alpha} I
\end{smallmatrix}\bigr ]$. Then, the exponential convergence of the function $\tilde{\mathcal{L}}$ with the rate $m>0$
\begin{align}
\tilde{\mathcal{L}}(\tilde{\rho}_{\tau})\leq e^{-m \tau}\tilde{\mathcal{L}}(\tilde{\rho}_{0})
\end{align}
implies
\begin{align}
\mathcal{L}(\rho_{t})\leq e^{-\sqrt{\alpha}mt}\mathcal{L}(\rho_{0}).
\end{align}
\end{corollary}
\begin{proof}
By the rescaling of variables $\tau = \sqrt{\alpha} t$, $\tilde{\mathbf{q}} = \mathbf{q}$, and $\tilde{\mathbf{p}}=\frac{1}{\sqrt{\alpha}}\mathbf{p}$, $\tilde{\rho}_{\tau}(\tilde{\mathbf{q}},\tilde{\mathbf{p}})=\alpha^{\frac{d}{2}}\rho_{t}(\mathbf{q}, \mathbf{p})$ and $\tilde{\pi}(\tilde{\mathbf{q}},\tilde{\mathbf{p}})=\alpha^{\frac{d}{2}}\pi(\mathbf{q}, \mathbf{p})$ holds by direct computation. By these properties, we have
\begin{align}
\chi^{2}(\tilde{\rho}_{\tau}\Vert\tilde{\pi}) &= \int (\frac{\tilde{\rho}_{\tau}(\tilde{\mathbf{q}}, \tilde{\mathbf{p}})}{\tilde{\pi}(\tilde{\mathbf{q}}, \tilde{\mathbf{p}})}-1)^{2}\tilde{\pi}(\tilde{\mathbf{q}},\tilde{\mathbf{p}})\dd\tilde{\mathbf{q}}\dd\tilde{\mathbf{p}}\\&=\int(\frac{\rho_{t}(\mathbf{q},\mathbf{p})}{\pi(\mathbf{q},\mathbf{p})}-1)^{2}\pi(\mathbf{q},\mathbf{p})\dd\mathbf{q}\dd\mathbf{p}=\chi^{2}(\rho_{t}\Vert\pi)
\end{align}
and similarly derive the equation
\begin{align}
\mathbb{E}_{\tilde{\pi}}\big[\,\big\langle\, \nabla_{\tilde{x}}\frac{\mathbf{\tilde{\rho}}_{\tau}}{\tilde{\pi}},\,\tilde{S}\,\nabla_{\tilde{x}}\frac{\mathbf{\tilde{\rho}}_{\tau}}{\tilde{\pi}}\,\big\rangle\,\big] = \mathbb{E}_{\pi}\big[\,\big\langle\, P\,\nabla_{\mathbf{x}}\frac{\rho_{t}}{\pi},\,\tilde{S}P\,\nabla_{\mathbf{x}}\frac{\rho_{t}}{\pi}\,\big\rangle\,\big].
\end{align}
Therefore, two Lyapunov functions are equal since
\begin{align}
    \tilde{\mathcal{L}}(\tilde{\rho}_{\tau})=&\chi^{2}(\tilde{\rho}_{\tau}\Vert\tilde{\pi})+\mathbb{E}_{\tilde{\pi}}\big[\,\big\langle\, \nabla_{\tilde{x}}\frac{\mathbf{\tilde{\rho}}_{\tau}}{\tilde{\pi}},\,\tilde{S}\,\nabla_{\tilde{x}}\frac{\mathbf{\tilde{\rho}}_{\tau}}{\tilde{\pi}}\,\big\rangle\,\big] \\=&\chi^{2}(\rho_{t}\Vert\pi) +\mathbb{E}_{\pi}\big[\,\big\langle\, P\,\nabla_{\mathbf{x}}\frac{\rho_{t}}{\pi},\,\tilde{S}P\,\nabla_{\mathbf{x}}\frac{\rho_{t}}{\pi}\,\big\rangle\,\big]=\mathcal{L}(\rho_{t}),
\end{align}
so the Lyapunov function $\mathcal{L}(\rho_{t})$ for original dynamics satisfies
\begin{align}
\mathcal{L}(\rho_{t})\leq e^{-\sqrt{\alpha}m t}\mathcal{L}(\rho_{0}).
\end{align}
\end{proof}

We introduce two assumptions for the original dynamics \cref{eq: original}. 
\begin{assumption}\label{assu: original} Regarding the original dynamics \cref{eq: original}, we introduce two assumptions for the friction coefficient $\Gamma = s\sqrt{\text{Hess}\, \, V}$ case.
\begin{enumerate}[label=(\alph*)]
\item Potential function $V$ is $\beta$-smooth and $\alpha$-strongly convex for some $\alpha$, $\beta>0$, i.e. $\alpha I\preceq \text{Hess}\, V\preceq \beta I.$ \label{assum1}
\item The derivative of the square root of Hessian $V$ with respect to any $\mathbf{q}_{i}$ is bounded, i.e. $\Vert\frac{\partial\sqrt{\text{Hess}\, \, V}}{\partial q_{i}}\Vert_{2} \leq \gamma$. \label{assum2}
\end{enumerate}
\end{assumption}
Then, we denote \cref{assu: rescaled} for the rescaled dynamics \cref{eq: rescaled} based on original \cref{assu: original} for the original dynamics \cref{eq: original}. 
\begin{assumption}\label{assu: rescaled} We introduce the assumptions for the rescaled dynamics \cref{eq: rescaled} with the friction coefficient $\Gamma = s\sqrt{\text{Hess}\, \, V}$.
\begin{enumerate}[label=(\alph*)]
\item Potential function $V$ is $\frac{\beta}{\alpha}$-smooth and $1$-strongly convex, i.e. $I\preceq \text{Hess}\, V\preceq \frac{\beta}{\alpha} I.$
\item The derivative of the square root of Hessian $V$ with respect to any $\mathbf{q}_{i}$ is bounded, i.e. $\Vert \frac{\partial\sqrt{\text{Hess}\, \, V}}{\partial q_{i}} \Vert_{2} \leq \frac{\gamma}{\sqrt{\alpha}} $.
\end{enumerate}
\end{assumption}

\section{Potential Function Related Friction Coefficient Case}\label{sec3}
In this section, we consider the rescaled dynamics \cref{eq: rescaled} with \cref{assu: rescaled} where the friction coefficient is given as $\Gamma = s\sqrt{\text{Hess}\, \, V}$ with $s>0$. First, we evaluate the derivative of the Lyapunov function and obtain its bound. Then, based on the bound of the derivative, we prove that Lyapunov function exponentially converges in certain conditions, which leads the convergence of $\chi^{2}$-divergence.

We analyze the Lyapunov function based on the approach of \cite{12}. In the next theorem, we compute the derivative of the Lyapunov function when the friction coefficient is $\Gamma=s\sqrt{\text{Hess}\, \, V}$. We note that this result holds for any $\Gamma$ given as function of $\mathbf{q}$, and it coincides with the result of \cite{12} since the terms with derivative of $\Gamma$ become zero when $\Gamma$ is constant.
\begin{theorem}\label{thm5}
Let Lyapunov function $\mathcal{L}(\rho_{t})=\chi^{2}(\rho_{t}\Vert \pi)+\mathbb{E}_{\pi}\big[\,\big\langle\, \nabla_{\mathbf{x}}\frac{\rho_{t}}{\pi},\,S\,\nabla_{\mathbf{x}}\frac{\rho_{t}}{\pi}\,\big\rangle\,\big]$, where $S$ is a positive semidefinite matrix
$S=
\bigl[ \begin{smallmatrix}b\Gamma^{-2}\,&a\Gamma^{-1}\,\\a\Gamma^{-1}\,&cI\,
\end{smallmatrix}\bigr]=\bigl[\begin{smallmatrix}
b(\Gamma)&a(\Gamma)\\
a(\Gamma)&c(\Gamma)
\end{smallmatrix}\bigr]$ with positive coefficients $a, b$, and $c$. Then, letting $h(\rho_{t})=\frac{\rho_{t}}{\pi}$, derivative of the Lyapunov function is given as
\begin{align}
\frac{\dd \mathcal{L}(\rho_{t})}{\dd t}=&-2\mathbb{E}_{\pi}\big\{\big\langle\,\nabla_{\mathbf{p}}h,(\alpha^{-1}\Gamma+c(\Gamma)\Gamma- a(\Gamma)\text{Hess}\, V)\nabla_{\mathbf{p}}h\,\big\rangle+\big\langle\,\nabla_{\mathbf{q}}h, a(\Gamma)\nabla_{\mathbf{q}}h\,\big\rangle\\&
+\big\langle\,\nabla_{\mathbf{p}}h,(a(\Gamma)\Gamma+c(\Gamma)-b(\Gamma)\text{Hess}\, V)\nabla_{\mathbf{q}}h\,\big\rangle+\alpha^{-1}\big\langle\,
SA,\nabla_{\mathbf{p}}\nabla_{\mathbf{x}}h\,\big\rangle_{F}\\&+\alpha^{-1}\big\langle\,
\nabla_{\mathbf{p}}\nabla_{\mathbf{x}}h\cdot\Gamma,S\nabla_{\mathbf{p}}\nabla_{\mathbf{x}}h\,\big\rangle_{F}
+\alpha^{-1}\big\langle\,B,\nabla_{\mathbf{p}}\nabla_{\mathbf{x}}h\,\big\rangle_{F}\big\},
\end{align}
where $A$ is a $2d \cross d$ matrix having the $i$-th row as
\begin{align}
(A)_{i}=(\frac{\partial \Gamma}{\partial x_{i}}\nabla_{\mathbf{p}}h)^{\top}
\end{align}
and $B$ is a $2d \cross d$ matrix that has the entries
\begin{align}
B_{ij}=([S,\nabla_{\mathbf{q}}]\nabla_{\mathbf{x}}h)_{ij} = -\sum_{k}\frac{\partial S_{ik}}{\partial q_{j}}\frac{\partial h}{\partial x_{k}}.
\end{align}
Here, we define the $d\cross d$ matrix $A_{1}$ and $A_{2}$ that $A_{1}$ is comprised of first $d$ rows of $A$ and $A_{2}$ is comprised of next $d$ rows of $A$, and we denote $A = \bigl[\begin{smallmatrix}
A_{1} \\ A_{2}
\end{smallmatrix}\bigr]$.
\end{theorem}
\begin{proof}
The proof is similar to the proof of \cite[Lemma 20]{12}. First, Fokker-Planck equation of the rescaled dynamics \cref{eq: rescaled} can be written as 
\begin{equation}
\frac{\partial\rho_{t}}{\partial t}+\nabla_{\mathbf{x}}\cdot(\rho_{t}\mathbf{J}_{t})=0 \text{ where }
\mathbf{J}_{t} = 
\begin{bmatrix}
\mathbf{p_{t}}\\-\nabla V(\mathbf{q}_{t})-\Gamma(\mathbf{q}_{t})\mathbf{p}_{t}-\alpha^{-1}\Gamma(\mathbf{q}_{t})\nabla_{\mathbf{p}}\log\rho_{t}
\end{bmatrix}.
\end{equation}
Here, by the property
\begin{equation}
\nabla_{\mathbf{x}}\cdot
\begin{bmatrix}
\rho_{t}
\begin{bmatrix}
-\nabla_{\mathbf{p}}\log\rho_{t}\\\nabla_{\mathbf{q}}\log\rho_{t}
\end{bmatrix}
\end{bmatrix}
=0,
\end{equation}
$\mathbf{J}_{t}$ can be converted to
\begin{equation}
\mathbf{J}_{t}
=\alpha^{-1}\frac{\pi}{\rho_{t}}
\begin{bmatrix}
O& I \\
-I& -\Gamma(\mathbf{q}_{t})
\end{bmatrix}
\begin{bmatrix}
\nabla_{\mathbf{q}}\frac{\rho_{t}}{\pi}\\\nabla_{\mathbf{p}}\frac{\rho_{t}}{\pi}
\end{bmatrix}
=-\alpha^{-1}\frac{\pi}{\rho_{t}}G(\mathbf{q}_{t})\nabla_{\mathbf{x}}\frac{\rho_{t}}{\pi}
\end{equation}
where $G(\mathbf{q}_{t})=
\begin{bmatrix}
O&-I\\
I&\Gamma(\mathbf{q}_{t})
\end{bmatrix}$.
Next, we analyze the time derivative of Lyapunov function. By direct computation,  
\begin{align}
    \frac{\dd}{\dd \varepsilon}\mathcal{L}[\rho+\varepsilon \phi]|_{\varepsilon = 0} &= \frac{\dd}{\dd \varepsilon}\big(\,\chi^{2}(\rho+\varepsilon\phi\Vert \pi)+\mathbb{E}_{\pi}\big[\,\big\langle\, \nabla_{\mathbf{x}}\frac{\rho+\varepsilon \phi}{\pi},\,S\,\nabla_{\mathbf{x}}\frac{\rho+\varepsilon\phi}{\pi}\,\big\rangle\,\big]\,\big)\big|_{\varepsilon = 0}\\
    &= \int \frac{2\rho\phi}{\pi}\dd x + \mathbb{E}_{\pi}\big[\,\big\langle\, \nabla_{\mathbf{x}}\frac{\rho}{\pi},\,S\,\nabla_{\mathbf{x}}\frac{\phi}{\pi}\,\big\rangle\,\big]+ \mathbb{E}_{\pi}\big[\,\big\langle\, \nabla_{\mathbf{x}}\frac{\phi}{\pi},\,S\,\nabla_{\mathbf{x}}\frac{\rho}{\pi}\,\big\rangle\,\big]\\
    &= \int \frac{2\rho\phi}{\pi}\dd x + 2\mathbb{E}_{\pi}\big[\,\big\langle\, \frac{\phi}{\pi},(\nabla_{\mathbf{x}})^{\ast}\,S\,\nabla_{\mathbf{x}}\frac{\rho}{\pi}\,\big\rangle\,\big]\\
    &= \int \big(\,\frac{2\rho}{\pi}+ 2(\nabla_{x})^{\ast}\, S\,\nabla_{\mathbf{x}}\frac{\rho}{\pi}\,\big) \,\phi\, \dd x,
\end{align}
and recalling $h(\rho_{t})= \rho_{t} / \pi$, we obtain the functional derivative of the Lyapunov function $\mathcal{L}(\rho_{t})$ as
\begin{equation}
\frac{\delta\mathcal{L}(\rho_{t})}{\delta\rho_{t}}=2h+2(\nabla_{\mathbf{x}})^{*}\,S\,\nabla_{\mathbf{x}}h.  
\end{equation}

Then, time derivative of the Lyapunov function is derived as 
\begin{align}
\frac{\dd}{\dd t}\mathcal{L}(\rho_{t}) &=\int\frac{\delta\mathcal{L}(\rho_{t})}{\delta\rho_{t}}\,\frac{\partial\rho_{t}}{\partial t}\,\dd x=\int\big\langle\,\nabla_{\mathbf{x}}\frac{\delta\mathcal{L}(\rho_{t})}{\delta\rho_{t}},\,\rho_{t}\mathbf{J}_{t}\,\big\rangle\,\dd x\\&=-\alpha^{-1}\int\big\langle\,\nabla_{\mathbf{x}}\frac{\delta\mathcal{L}(\rho_{t})}{\delta\rho_{t}},\,
G\,\nabla_{\mathbf{x}}\frac{\rho_{t}}{\pi}\,\big\rangle\,\dd\pi\\&=-2\alpha^{-1}\int\big\langle\,\nabla_{\mathbf{x}}h+\nabla_{\mathbf{x}}(\nabla_{\mathbf{x}})^{*}(S\nabla_{\mathbf{x}}h),G\nabla_{\mathbf{x}}h\,\big\rangle\,\dd \pi.
\end{align}

Lastly, we analyze the term $\nabla_{\mathbf{x}}(\nabla_{\mathbf{x}})^{*}(S\nabla_{\mathbf{x}}h)$ and derive the final result. The above term $\nabla_{\mathbf{x}}(\nabla_{\mathbf{x}})^{*}(S\nabla_{\mathbf{x}}h)$ is evaluated as
\begin{align}
\nabla_{\mathbf{x}}(\nabla_{\mathbf{x}})^{*}(S\nabla_{\mathbf{x}}h)
=\nabla_{\mathbf{x}}\begin{bmatrix}
(\nabla_{\mathbf{q}})^{*}\\
(\nabla_{\mathbf{p}})^{*}
\end{bmatrix}
\begin{bmatrix}
b(\Gamma)&a(\Gamma)\\
a(\Gamma)&c(\Gamma)
\end{bmatrix}
\begin{bmatrix}
\nabla_{\mathbf{q}} h\\
\nabla_{\mathbf{p}} h
\end{bmatrix}=\nabla_{\mathbf{x}}f,
\end{align}
where 
\begin{equation}\label{eq3}
f=(\nabla_{\mathbf{q}})^{*}b(\Gamma)\nabla_{\mathbf{q}}h+(\nabla_{\mathbf{q}})^{*}a(\Gamma)\nabla_{\mathbf{p}}h+(\nabla_{\mathbf{p}})^{*}a(\Gamma)\nabla_{\mathbf{q}}h+(\nabla_{\mathbf{p}})^{*}c(\Gamma)\nabla_{\mathbf{p}}h.
\end{equation}
Then, since
\begin{align}
\mathbb{E}_{\pi}\big[\big\langle\nabla_{\mathbf{x}}f, G\nabla_{\mathbf{x}}h\big\rangle\big] =\mathbb{E}_{\pi}[-\big\langle\,\nabla_{\mathbf{q}}f, \nabla_{\mathbf{p}}h\,\big\rangle+\big\langle\,\nabla_{\mathbf{p}}f, \nabla_{\mathbf{q}}h\,\big\rangle+\big\langle\,\nabla_{\mathbf{p}}f,\Gamma\nabla_{\mathbf{p}}h\,\big\rangle]
\end{align}
holds, by \cref{thm6} we obtain 
\begin{align}
\frac{\dd \mathcal{L}(\rho_{t})}{\dd t}=&-2\mathbb{E}_{\pi}\big\{\alpha^{-1}\big\langle\,\nabla_{\mathbf{x}}h, G\nabla_{\mathbf{x}}h\,\big\rangle+\big\langle\,\nabla_{\mathbf{p}}h,(c(\Gamma)\Gamma- a(\Gamma)\text{Hess}\, V+\Gamma)\nabla_{\mathbf{p}}h\,\big\rangle\\&
+\big\langle\,\nabla_{\mathbf{p}}h,(a(\Gamma)\Gamma+c(\Gamma)-b(\Gamma)\text{Hess}\, V)\nabla_{\mathbf{q}}h\,\big\rangle+\big\langle\,\nabla_{\mathbf{q}}h, a(\Gamma)\nabla_{\mathbf{q}}h\,\big\rangle\\&
+\alpha^{-1}\big\langle\,
\begin{bmatrix}
\nabla_{\mathbf{q}}(\nabla_{\mathbf{p}})^{*}\Gamma\nabla_{\mathbf{p}}h\\
\nabla_{\mathbf{p}}(\nabla_{\mathbf{p}})^{*}\Gamma\nabla_{\mathbf{p}}h-\alpha\Gamma\nabla_{\mathbf{p}}h
\end{bmatrix}
,S\nabla_{\mathbf{x}}h\,\big\rangle\\&+\alpha^{-1}
\big\langle\,\nabla_{\mathbf{p}}\nabla_{\mathbf{x}}h, [S,\nabla_{\mathbf{q}}]\nabla_{\mathbf{x}}h\,\big\rangle_{F}\big\}. 
\end{align}
Here, by \cref{thm9},
\begin{equation}
\big\langle\,
\begin{bmatrix}
\nabla_{\mathbf{q}}(\nabla_{\mathbf{p}})^{*}\Gamma\nabla_{\mathbf{p}}h\\
\nabla_{\mathbf{p}}(\nabla_{\mathbf{p}})^{*}\Gamma\nabla_{\mathbf{p}}h-\alpha\Gamma\nabla_{\mathbf{p}}h
\end{bmatrix}
,S\nabla_{\mathbf{x}}h\,\big\rangle = \big\langle\,\nabla_{\mathbf{p}}\nabla_{\mathbf{x}}h\cdot\Gamma,S\nabla_{\mathbf{p}}\nabla_{\mathbf{x}}h\,\big\rangle_{F}+\big\langle\,A,S\nabla_{\mathbf{p}}\nabla_{\mathbf{x}}h\,\big\rangle_{F}
\end{equation}
holds, so sum of the last two terms are given by
\begin{align}
\big\langle\,&
\begin{bmatrix}
\nabla_{\mathbf{q}}(\nabla_{\mathbf{p}})^{*}\Gamma\nabla_{\mathbf{p}}h\\
\nabla_{\mathbf{p}}(\nabla_{\mathbf{p}})^{*}\Gamma\nabla_{\mathbf{p}}h-\alpha\Gamma\nabla_{\mathbf{p}}h
\end{bmatrix}
,S\nabla_{\mathbf{x}}h\,\big\rangle+
\big\langle\,\nabla_{\mathbf{p}}\nabla_{\mathbf{x}}h, [S,\nabla_{\mathbf{q}}]\nabla_{\mathbf{x}}h\,\big\rangle_{F}\\=&\big\langle\,
\nabla_{\mathbf{p}}\nabla_{\mathbf{x}}h\cdot\Gamma,S\nabla_{\mathbf{p}}\nabla_{\mathbf{x}}h\,\big\rangle_{F}+\big\langle\,
SA,\nabla_{\mathbf{p}}\nabla_{\mathbf{x}}h\,\big\rangle_{F}+\big\langle\,B,\nabla_{\mathbf{p}}\nabla_{\mathbf{x}}h\,\big\rangle_{F}
\end{align}
and the proof is complete.
\end{proof}
\begin{lemma}\label{thm6} Recalling the notions of the proof of \cref{thm5}, the two equation
\begin{align}
-\big\langle\,\nabla_{\mathbf{q}}f,\nabla_{\mathbf{p}}h\,\big\rangle+\big\langle\,\nabla_{\mathbf{p}}f,\nabla_{\mathbf{q}}h\,\big\rangle=&\alpha\big\langle\,\nabla_{\mathbf{p}}h,-a(\Gamma)\text{Hess}\,V\nabla_{\mathbf{p}}h\,\big\rangle+\alpha\big\langle\,\nabla_{\mathbf{q}}h,a(\Gamma)\nabla_{\mathbf{q}}h\,\big\rangle\\+&\alpha\big\langle\,\nabla_{\mathbf{p}}h,(c(\Gamma)-b(\Gamma)\text{Hess}\,V)\nabla_{\mathbf{q}}h\,\big\rangle
\\+&\big\langle\,\nabla_{\mathbf{p}}\nabla_{\mathbf{x}}h,[S,\nabla_{\mathbf{q}}]\nabla_{\mathbf{x}}h\,\big\rangle_{F}
\end{align}
and
\begin{align}
\big\langle\,\nabla_{\mathbf{p}}f,\Gamma\nabla_{\mathbf{p}}h\,\big\rangle=\big\langle\,\nabla_{\mathbf{x}}(\nabla_{\mathbf{p}})^{*}\Gamma\nabla_{\mathbf{p}}h, S\nabla_{\mathbf{x}}h\,\big\rangle
\end{align}
holds.
\end{lemma}
\begin{proof}
Note that $[\nabla_{\mathbf{p}}, B][h] =\alpha\nabla_{\mathbf{q}}h$ and $[\nabla_{\mathbf{q}}, B][h] =-\alpha\text{Hess}\,V\nabla_{\mathbf{p}}h$. For the first equation,
\begin{align}
\big\langle\,f, B[h]\,\big\rangle=-\big\langle\,\nabla_{\mathbf{q}}f, \nabla_{\mathbf{p}}h\,\big\rangle+\big\langle\,\nabla_{\mathbf{p}}f, \nabla_{\mathbf{q}}h\,\big\rangle 
\end{align}
and
\begin{align}
  \big\langle\,f, B[h]\,\big\rangle =& \big\langle\,b\nabla_{\mathbf{q}}h+a\nabla_{\mathbf{p}}h,\nabla_{\mathbf{q}}B[h]\,\big\rangle+\big\langle\,a\nabla_{\mathbf{q}}h+c\nabla_{\mathbf{p}}h, \nabla_{\mathbf{p}}B[h]\,\big\rangle\\=& \big\langle\,b\nabla_{\mathbf{q}}h+a\nabla_{\mathbf{p}}h,B[\nabla_{\mathbf{q}}h]\,\big\rangle+\big\langle\,a\nabla_{\mathbf{q}}h+c\nabla_{\mathbf{p}}h, B[\nabla_{\mathbf{p}}h]\,\big\rangle\\&
+\big\langle\,b\nabla_{\mathbf{q}}h+a\nabla_{\mathbf{p}}h,-\alpha\text{Hess}\,V\nabla_{\mathbf{p}}h\,\big\rangle+\big\langle\,a\nabla_{\mathbf{q}}h+c\nabla_{\mathbf{p}}h,\alpha\nabla_{\mathbf{q}}h\,\big\rangle
\end{align}
holds. In addition, since
\begin{align}
\big\langle\,&b\nabla_{\mathbf{q}}h+a\nabla_{\mathbf{p}}h,B[\nabla_{\mathbf{q}}h]\,\big\rangle+\big\langle\,a\nabla_{\mathbf{q}}h+c\nabla_{\mathbf{p}}h, B[\nabla_{\mathbf{p}}h]\,\big\rangle\\=&
\big\langle\,b\nabla_{\mathbf{q}}h+a\nabla_{\mathbf{p}}h,(\nabla_{\mathbf{p}})^{*}\nabla_{\mathbf{q}}\nabla_{\mathbf{q}}h-(\nabla_{\mathbf{q}})^{*}\nabla_{\mathbf{p}}\nabla_{\mathbf{q}}h\,\big\rangle\\&+\big\langle\,a\nabla_{\mathbf{q}}h+c\nabla_{\mathbf{p}}h,(\nabla_{\mathbf{p}})^{*}\nabla_{\mathbf{q}}\nabla_{\mathbf{p}}h-(\nabla_{\mathbf{q}})^{*}\nabla_{\mathbf{p}}\nabla_{\mathbf{p}}h\,\big\rangle\\
=&\big\langle\,\nabla_{\mathbf{p}}\nabla_{\mathbf{x}}h, [S,\nabla_{\mathbf{q}}]\nabla_{\mathbf{x}}h\,\big\rangle_{F},
\end{align}
the proof of the first equation is complete. Second equation holds since
\begin{align}
\big\langle\,\nabla_{\mathbf{p}}f,\Gamma\nabla_{\mathbf{p}}h\,\big\rangle&=\big\langle\,f,(\nabla_{\mathbf{p}})^{*}\Gamma\nabla_{\mathbf{p}}h\,\big\rangle\\&=\big\langle\,(\nabla_{\mathbf{x}})^{*}S\nabla_{\mathbf{x}}h,(\nabla_{\mathbf{p}})^{*}\Gamma\nabla_{\mathbf{p}}h\,\big\rangle=\big\langle\,\nabla_{\mathbf{x}}(\nabla_{\mathbf{p}})^{*}\Gamma\nabla_{\mathbf{p}}h, S\nabla_{\mathbf{x}}h\,\big\rangle.
\end{align}
\end{proof}
\begin{lemma}\label{thm9}
Recalling the notions of the proof of \cref{thm5},
the following result holds.
\begin{align}\label{eq4}
\begin{bmatrix}
\nabla_{\mathbf{q}}(\nabla_{\mathbf{p}})^{*}\Gamma\nabla_{\mathbf{p}}h\\
\nabla_{\mathbf{p}}(\nabla_{\mathbf{p}})^{*}\Gamma\nabla_{\mathbf{p}}h-\alpha\Gamma\nabla_{\mathbf{p}}h
\end{bmatrix} = (\nabla_{\mathbf{p}})^{*}(\nabla_{\mathbf{p}}\nabla_{\mathbf{x}}h\cdot\Gamma)+(\nabla_{\mathbf{p}})^{*}A    
\end{align}
\end{lemma}
\begin{proof}
Since
\begin{align}
(\nabla_{\mathbf{p}})^{*}\Gamma\nabla_{\mathbf{p}}h = \alpha\mathbf{p}^{\top}\cdot\Gamma\nabla_{\mathbf{p}}h-\sum_{i}\frac{\partial (\Gamma\nabla_{\mathbf{p}}h)_{i}}{\partial p_{i}}=\alpha\mathbf{p}^{\top}\cdot\Gamma\nabla_{\mathbf{p}}h-\big\langle\,\Gamma,\nabla_{\mathbf{p}}\nabla_{\mathbf{p}}h\,\big\rangle_{F},
\end{align}
two equations
\begin{align}
\nabla_{\mathbf{q}}(\nabla_{\mathbf{p}})^{*}\Gamma\nabla_{\mathbf{p}}h=\nabla_{\mathbf{q}}(\alpha\mathbf{p}^{\top}\cdot\Gamma\nabla_{\mathbf{p}}h)-\nabla_{\mathbf{q}}\big\langle\,\Gamma,\nabla_{\mathbf{p}}\nabla_{\mathbf{p}}h\,\big\rangle_{F} 
\end{align}
and
\begin{align}
\nabla_{\mathbf{p}}(\nabla_{\mathbf{p}})^{*}\Gamma\nabla_{\mathbf{p}}h-\alpha\Gamma\nabla_{\mathbf{p}}h=\nabla_{\mathbf{p}}(\alpha\mathbf{p}^{\top}\cdot\Gamma\nabla_{\mathbf{p}}h)-\nabla_{\mathbf{p}}\big\langle\,\Gamma,\nabla_{\mathbf{p}}\nabla_{\mathbf{p}}h\,\big\rangle_{F}-\alpha\Gamma\nabla_{\mathbf{p}}h
\end{align}
holds. First, we consider the left side of desired equation \cref{eq4}. For the term $\nabla_{\mathbf{q}}(\nabla_{\mathbf{p}})^{*}\Gamma\nabla_{\mathbf{p}}h$, recalling the notions of \cref{thm5},
\begin{align}
(\nabla_{\mathbf{q}}(\alpha\mathbf{p}^{\top}\cdot\Gamma\nabla_{\mathbf{p}}h))_{i} =& \alpha\frac{\partial}{\partial q_{i}}\big\{\sum_{j,k}\mathbf{p}_{j}\Gamma_{jk}\frac{\partial h}{\partial p_{k}}\big\}=\alpha\sum_{j,k}\mathbf{p}_{j}\Gamma_{jk}\frac{\partial^{2} h}{\partial p_{k}\partial q_{i}}+\alpha\sum_{j,k}\mathbf{p}_{j}\frac{\Gamma_{jk}}{\partial q_{i}}\frac{\partial h}{\partial p_{k}}\\=&\alpha(\nabla_{\mathbf{p}}\nabla_{\mathbf{q}}h\cdot\Gamma\cdot\mathbf{p})_{i}+\alpha(A_{1}\cdot\mathbf{p})_{i}
\end{align}
and
\begin{align}
(\nabla_{\mathbf{q}}\big\langle\,\Gamma,\nabla_{\mathbf{p}}\nabla_{\mathbf{p}}h\,\big\rangle_{F})_{i} =\sum_{j,k}\Gamma_{jk}\frac{\partial^{3}h}{\partial q_{i}\partial p_{j}\partial p_{k}}+ \sum_{j,k}\frac{\partial \Gamma_{jk}}{\partial q_{i}}\frac{\partial^{2}h}{\partial p_{j}\partial p_{k}}
\end{align}
holds. In addition, for the term 
$\nabla_{\mathbf{p}}(\nabla_{\mathbf{p}})^{*}\Gamma\nabla_{\mathbf{p}}h-\alpha\Gamma\nabla_{\mathbf{p}}h$,
\begin{align}
(\nabla_{\mathbf{p}}(\alpha\mathbf{p}^{\top}\cdot\Gamma\nabla_{\mathbf{p}}h)-\alpha\Gamma\nabla_{\mathbf{p}}h)_{i} =& \alpha(\Gamma\nabla_{\mathbf{p}}h)_{i}+\alpha\sum_{j}\mathbf{p}_{j}\frac{\partial (\Gamma\nabla_{\mathbf{p}}h)_{j}}{\partial p_{i}}-\alpha(\Gamma\nabla_{\mathbf{p}}h)_{i}\\=&\alpha\sum_{j,k}\Gamma_{jk}\frac{\partial^{2} h}{\partial p_{i}\partial p_{k}}\mathbf{p}_{j}=\alpha(\nabla_{\mathbf{p}}\nabla_{\mathbf{p}}h\cdot\Gamma\cdot\mathbf{p})_{i}+\alpha(A_{2}\cdot\mathbf{p})_{i}
\end{align}
and 
\begin{align}
(\nabla_{\mathbf{p}}\big\langle\,\Gamma,\nabla_{\mathbf{p}}\nabla_{\mathbf{p}}h\,\big\rangle_{F})_{i} = \sum_{j,k}\Gamma_{jk}\frac{\partial^{3}h}{\partial p_{i}\partial p_{j}\partial p_{k}}
\end{align}
holds. Next, we compare these terms with the right side of the desired equation \cref{eq4}. Since
\begin{align}
((\nabla_{\mathbf{p}})^{*}(\nabla_{\mathbf{p}}\nabla_{\mathbf{q}}h\cdot\Gamma))_{i} =& \alpha\sum_{j}\mathbf{p}_{j}(\nabla_{\mathbf{p}}\nabla_{\mathbf{q}}h\cdot\Gamma)_{ij}- \sum_{j}\frac{\partial(\nabla_{\mathbf{p}}\nabla_{\mathbf{q}}h\cdot\Gamma)_{ij}}{\partial p_{j}}\\=&\alpha(\nabla_{\mathbf{p}}\nabla_{\mathbf{q}}h\cdot\Gamma\cdot \mathbf{p})_{i}- \sum_{j,k}\frac{\partial^{3}h}{\partial p_{j}\partial p_{k}\partial q_{i}}\Gamma_{kj},
\end{align}
and
\begin{align}
((\nabla_{\mathbf{p}})^{*}(\nabla_{\mathbf{p}}\nabla_{\mathbf{p}}h\cdot\Gamma))_{i} =& \alpha\sum_{j}\mathbf{p}_{j}(\nabla_{\mathbf{p}}\nabla_{\mathbf{p}}h\cdot\Gamma)_{ij}- \sum_{j}\frac{\partial(\nabla_{\mathbf{p}}\nabla_{\mathbf{p}}h\cdot\Gamma)_{ij}}{\partial p_{j}}\\=&\alpha(\nabla_{\mathbf{p}}\nabla_{\mathbf{p}}h\cdot\Gamma\cdot \mathbf{p})_{i}- \sum_{j,k}\frac{\partial^{3}h}{\partial p_{j}\partial p_{k}\partial p_{i}}\Gamma_{kj},  
\end{align}
and
\begin{align}
((\nabla_{\mathbf{p}})^{*}A_{1})_{i}=\alpha\sum_{j}\mathbf{p}_{j}(A_{1})_{ij}- \sum_{j}\frac{\partial(A_{1})_{ij}}{\partial p_{j}}=\alpha(A_{1}\cdot\mathbf{p})_{i}- \sum_{j,k}\frac{\partial \Gamma_{jk}}{\partial q_{i}}\frac{\partial^{2}h}{\partial p_{j}\partial p_{k}}
\end{align}
holds, we prove that the left side and the right side of \cref{eq4} are equal.
\end{proof}
Now, applying the conditions $\Gamma = s\sqrt{\text{Hess}\, \, V}$ and \cref{assu: rescaled} to \cref{thm5}, we obtain the upper bound of the derivative of Lyapunov function as below. We prove the exponential convergence of the Lyapunov function based on this bound. We note that this part is different from the previous research, since $A$ and $B$ become zero when the $\Gamma$ is constant.
\begin{theorem}\label{cor1} We denote two constants $L_{1} =  \frac{bd\alpha+cds^{2}\beta}{4(bc-a^{2})\alpha}(8b^{2}s^{-5}\alpha^{-1}+a^{2}s^{-3}\alpha^{-1})$ and $L_{2}= \frac{1}{4} bds^{-1} \alpha^{-1}+\frac{bd\alpha+cds^{2}\beta}{2(bc-a^{2})}a^{2}s^{-3}\alpha^{-2}$. Then,
\begin{align}
\frac{\dd \mathcal{L}(\rho_{t})}{\dd t}\leq&-2\mathbb{E}_{\pi}\big\{\big\langle\,\nabla_{\mathbf{p}}h,(\alpha^{-1}\Gamma+c(\Gamma)\Gamma- a(\Gamma)\text{Hess}\, V-L_{2}\gamma^{2}I)\nabla_{\mathbf{p}}h\,\big\rangle\\&+\big\langle\,\nabla_{\mathbf{p}}h,(a(\Gamma)\Gamma+c(\Gamma)-b(\Gamma)\text{Hess}\, V)\nabla_{\mathbf{q}}h\,\big\rangle\\&+\big\langle\,\nabla_{\mathbf{q}}h, (a(\Gamma)-L_{1}\gamma^{2}I)\nabla_{\mathbf{q}}h\,\big\rangle\big\}
\end{align}
holds.
\end{theorem}
\begin{proof}
We recall the result from the proof of \cref{thm5}, which is
\begin{align}
\big\langle\,&
\begin{bmatrix}
\nabla_{\mathbf{q}}(\nabla_{\mathbf{p}})^{*}\Gamma\nabla_{\mathbf{p}}h\\
\nabla_{\mathbf{p}}(\nabla_{\mathbf{p}})^{*}\Gamma\nabla_{\mathbf{p}}h-\alpha\Gamma\nabla_{\mathbf{p}}h
\end{bmatrix}
,S\nabla_{\mathbf{x}}h\,\big\rangle+
\big\langle\,\nabla_{\mathbf{p}}\nabla_{\mathbf{x}}h, [S,\nabla_{\mathbf{q}}]\nabla_{\mathbf{x}}h\,\big\rangle_{F}\\&=\big\langle\,
\nabla_{\mathbf{p}}\nabla_{\mathbf{x}}h\cdot\Gamma,S\nabla_{\mathbf{p}}\nabla_{\mathbf{x}}h\,\big\rangle_{F}+\big\langle\,
SA,\nabla_{\mathbf{p}}\nabla_{\mathbf{x}}h\,\big\rangle_{F}+\big\langle\,B,\nabla_{\mathbf{p}}\nabla_{\mathbf{x}}h\,\big\rangle_{F}.
\end{align}
First, since the trace of the product of two positive semidefinite matrices is nonnegative, we obtain 
\begin{equation}
\big\langle\,
\nabla_{\mathbf{p}}\nabla_{\mathbf{x}}h\cdot\Gamma,S\nabla_{\mathbf{p}}\nabla_{\mathbf{x}}h\,\big\rangle_{F}\geq s\big\langle\,
\nabla_{\mathbf{p}}\nabla_{\mathbf{x}}h,S\nabla_{\mathbf{p}}\nabla_{\mathbf{x}}h\,\big\rangle_{F}, 
\end{equation}
so
\begin{align}
\big\langle\,&
\begin{bmatrix}
\nabla_{\mathbf{q}}(\nabla_{\mathbf{p}})^{*}\Gamma\nabla_{\mathbf{p}}h\\
\nabla_{\mathbf{p}}(\nabla_{\mathbf{p}})^{*}\Gamma\nabla_{\mathbf{p}}h-\alpha\Gamma\nabla_{\mathbf{p}}h
\end{bmatrix}
,S\nabla_{\mathbf{x}}h\,\big\rangle+
\big\langle\,\nabla_{\mathbf{p}}\nabla_{\mathbf{x}}h, [S,\nabla_{\mathbf{q}}]\nabla_{\mathbf{x}}h\,\big\rangle_{F}\\
&\geq s\big\langle\,
\nabla_{\mathbf{p}}\nabla_{\mathbf{x}}h,S\nabla_{\mathbf{p}}\nabla_{\mathbf{x}}h\,\big\rangle_{F}+\big\langle\,
SA,\nabla_{\mathbf{p}}\nabla_{\mathbf{x}}h\,\big\rangle_{F}+\big\langle\,B,\nabla_{\mathbf{p}}\nabla_{\mathbf{x}}h\,\big\rangle_{F}\\
&\geq-\frac{1}{4s}\big\langle\,A, SA\,\big\rangle_{F}-\frac{1}{4s}\big\langle\,B,S^{-1}B\,\big\rangle_{F} 
\end{align}
holds. Next, we will analyze the two terms $\big\langle\,A,SA\,\big\rangle_{F}$, $\big\langle\,B,S^{-1}B\,\big\rangle_{F}$ and obtain the upper bound.
For the term $\big\langle\,A,SA\,\big\rangle_{F}$, letting $\mathbf{1}\in\mathbb{R}^{d}$ as matrix of ones,
\begin{align}
\big\langle\,A,SA\,\big\rangle_{F}=&\big\langle\,A^{\top},A^{\top}S^{\top}\,\big\rangle_{F}=\sum_{i,j}S_{ij}\big\langle\,\frac{\partial\Gamma}{\partial x_{i}}\nabla_{\mathbf{p}}h,\frac{\partial\Gamma}{\partial x_{j}}\nabla_{\mathbf{p}}h\,\big\rangle\\
\leq&s^{2}\alpha^{-1}\gamma^{2}\sum_{1\leq i,j\leq d}S_{ij}\big\langle\,\nabla_{\mathbf{p}}h,\nabla_{\mathbf{p}}h\,\big\rangle\leq bd \alpha^{-1}\gamma^{2}\big\langle\,\nabla_{\mathbf{p}}h,\nabla_{\mathbf{p}}h\,\big\rangle
\end{align}
since
\begin{align}
\sum_{1\leq i,j\leq d}S_{ij} = b\mathbf{1}^{\top}\Gamma^{-2}\mathbf{1}\leq bds^{-2}
\end{align}
holds. For the term $\big\langle\,B, S^{-1}B\,\big\rangle_{F}$,
\begin{align}
\big\langle\,B, S^{-1}B\,\big\rangle_{F} = \sum_{j}\big\langle\,\frac{\partial S}{\partial q_{j}}\nabla_{\mathbf{x}}h, S^{-1}\frac{\partial S}{\partial q_{j}}\nabla_{\mathbf{x}}h\,\big\rangle
\end{align}
holds, and since $S^{-1} = \frac{1}{bc-a^{2}}\bigl[\begin{smallmatrix}
c\Gamma^{2} & -a\Gamma\\ -a\Gamma & bI
\end{smallmatrix}\bigr]$
and
\begin{align}
\frac{\partial S}{\partial q_{j}}=
\begin{bmatrix}
-b(\Gamma^{-1}\frac{\partial \Gamma}{\partial q_{j}}\Gamma^{-2}+\Gamma^{-2}\frac{\partial \Gamma}{\partial q_{j}}\Gamma^{-1}) & -a\Gamma^{-1}\frac{\partial \Gamma}{\partial q_{j}}\Gamma^{-1}\\-a\Gamma^{-1}\frac{\partial \Gamma}{\partial q_{j}}\Gamma^{-1}& O
\end{bmatrix},
\end{align}
letting $X = -b(\Gamma^{-1}\frac{\partial \Gamma}{\partial q_{j}}\Gamma^{-2}+\Gamma^{-2}\frac{\partial \Gamma}{\partial q_{j}}\Gamma^{-1})$ and $Y=-a\Gamma^{-1}\frac{\partial \Gamma}{\partial q_{j}}\Gamma^{-1}$, we obtain
\begin{align}
\big\langle\,\frac{\partial S}{\partial q_{j}}\nabla_{\mathbf{x}}h,& S^{-1}\frac{\partial S}{\partial q_{j}}\nabla_{\mathbf{x}}h\,\big\rangle\\\leq& \Vert S^{-1} \Vert\big\langle\,\frac{\partial S}{\partial q_{j}}\nabla_{\mathbf{x}}h, \frac{\partial S}{\partial q_{j}}\nabla_{\mathbf{x}}h\,\big\rangle\leq\frac{b\alpha+cs^{2}\beta}{(bc-a^{2})\alpha}\big\langle\,\frac{\partial S}{\partial q_{j}}\nabla_{\mathbf{x}}h, \frac{\partial S}{\partial q_{j}}\nabla_{\mathbf{x}}h\,\big\rangle\\
=&\frac{b\alpha+cs^{2}\beta}{(bc-a^{2})\alpha}\big\{\big\langle\,X\nabla_{\mathbf{q}}h+Y\nabla_{\mathbf{p}}h, X\nabla_{\mathbf{q}}h+Y\nabla_{\mathbf{p}}h\,\big\rangle+\big\langle\,Y\nabla_{\mathbf{q}}h, Y\nabla_{\mathbf{q}}h\,\big\rangle\big\}\\
\leq&\frac{b\alpha+cs^{2}\beta}{(bc-a^{2})\alpha}\big\{2\big\langle\,X\nabla_{\mathbf{q}}h, X\nabla_{\mathbf{q}}h\,\big\rangle+2\big\langle\,Y\nabla_{\mathbf{p}}h,Y\nabla_{\mathbf{p}}h\,\big\rangle+\big\langle\,Y\nabla_{\mathbf{q}}h, Y\nabla_{\mathbf{q}}h\,\big\rangle\big\}\\
\leq&\frac{b\alpha+cs^{2}\beta}{(bc-a^{2})\alpha}\big\{(2\Vert X \Vert^{2}+\Vert Y \Vert ^{2})\big\langle\,\nabla_{\mathbf{q}}h, \nabla_{\mathbf{q}}h\,\big\rangle+2\Vert Y \Vert^{2}\big\langle\,\nabla_{\mathbf{p}}h,\nabla_{\mathbf{p}}h\,\big\rangle\big\}.
\end{align}
Here, by the properties of the matrix norm, 
\begin{align}
\Vert X \Vert\leq b(\Vert\Gamma^{-1}\frac{\partial \Gamma}{\partial q_{j}}\Gamma^{-2}\Vert+\Vert\Gamma^{-2}\frac{\partial \Gamma}{\partial q_{j}}\Gamma^{-1}\Vert)\leq 2 bs^{-2}\alpha^{-\frac{1}{2}}\gamma \,\,\text{ and }\,\,
\Vert Y \Vert &\leq as^{-1}\alpha^{-\frac{1}{2}}\gamma.
\end{align}
Therefore,
\begin{align}
\big\langle\,B,S^{-1}B\,\big\rangle_{F}\leq& \frac{bd\alpha+cds^{2}\beta}{(bc-a^{2})\alpha}\big\{2a^{2}s^{-2}\alpha^{-1}\gamma^{2}\big\langle\,\nabla_{\mathbf{p}}h,\nabla_{\mathbf{p}}h\,\big\rangle\\&+(8b^{2}s^{-4}\alpha^{-1}\gamma^{2}+a^{2}s^{-2}\alpha^{-1}\gamma^{2})\big\langle\,\nabla_{\mathbf{q}}h, \nabla_{\mathbf{q}}h\,\big\rangle\big\}
\end{align}
holds. Combining these results with  the proof is complete.
\end{proof}
Now, we prove the convergence of the Lyapunov function based on the bound of its derivative. We obtain the bound of the derivative of the Lyapunov function with Lyapunov function itself based on the Poincar\'e's inequality, then prove the exponential convergence.

\begin{theorem}\label{thm13} Recall the Lyapunov function $\mathcal{L}(\rho_{t})$ in Theorem \cref{thm5}, and assume that 
\begin{equation}
    a+c-bs^{-2} = 0,\, \alpha^{-1}+c-as^{-2}>0
\end{equation}
holds. Then, define the function $f$ and $g$ as
\begin{align}
    f(x) \coloneqq \frac{2 ax(\alpha^{-1}+ c-as^{-2})}{(\alpha^{-1}+bx^{2})(\alpha^{-1}+c-as^{-2})+(\alpha^{-1}+c)ax^{2}}
\end{align}
and
\begin{align}
    g(x) \coloneqq \frac{2((\alpha^{-1}+c)L_{1}+(\alpha^{-1}+ bx^{2})L_{2})}{(\alpha^{-1}+bx^{2})(\alpha^{-1}+ c)-a^{2}x^{2}},
\end{align}
and denote $m_{1}= \min \{f(s^{-1}\sqrt{\frac{\alpha}{\beta}}),f(s^{-1})\}$ and $m_{2} = \max \{g(s^{-1}\sqrt{\frac{\alpha}{\beta}}),g(s^{-1})\}$.

Then, when $\gamma$ satisfies $\gamma < \sqrt{\frac{m_{1}}{m_{2}}}$, the Lyapunov function exponentially converges to zero and satisfies
\begin{align}
\mathcal{L}(\rho_{t})\leq e^{-(m_{1}-\gamma^{2} m_{2})t}\mathcal{L}(\rho_{0}).
\end{align}

\end{theorem}
\begin{proof}
First, we evaluate the time derivative of the Lyapunov function. By direct computation and \cref{cor1}, we obtain
\begin{align}
\frac{\dd \mathcal{L}(\rho_{t})}{\dd t} \leq -\mathbb{E}_{\pi}\big[\,\big\langle\, \nabla_{\mathbf{x}}\frac{\rho_{t}}{\pi},\,(M_{1}+M_{2}+M_{3})\,\nabla_{\mathbf{x}}\frac{\rho_{t}}{\pi}\,\big\rangle\,\big]
\end{align}
where
\begin{align}
M_{1} = 
\begin{bmatrix}
O & O \\ O & 2\alpha^{-1}\Gamma
\end{bmatrix},\,\,
M_{3} = 
\begin{bmatrix}
-2L_{1}\gamma^{2}I& O\\O & -2L_{2}\gamma^{2}I
\end{bmatrix}
\end{align}
and
\begin{align}
M_{2} = 
\begin{bmatrix}
2a\Gamma^{-1} & (a+c)I-b\Gamma^{-2}\text{Hess}\,V\\
(a+c)I-b\Gamma^{-2}\text{Hess}\,V & 2c\Gamma - 2a\Gamma^{-1}\text{Hess}\,V
\end{bmatrix}
=\begin{bmatrix}
2a\Gamma^{-1}& O\\
O & 2(c-as^{-2})\Gamma  
\end{bmatrix}.
\end{align}
\par Next, in order to apply the Poincar\'e's inequality and the Gronwall's inequality, we denote the matrix 
$E = \alpha^{-1}I+S$ and find a constant $k_{1}$ such that $M_{1}+M_{2}\succeq k_{1} E$. We analyze this matrix
\begin{align}
M_{1}+M_{2}-&k_{1}E = M_{1}+M_{2} - k_{1}(\alpha^{-1}I+S)\\&=
\begin{bmatrix}
2a\Gamma^{-1}-k_{1}\alpha^{-1}I-bk_{1}\Gamma^{-2} & -ak_{1}\Gamma^{-1}\\
-ak_{1}\Gamma^{-1} &2(c-as^{-2}+\alpha^{-1})\Gamma-k_{1}(\alpha^{-1}+c)I
\end{bmatrix}
\end{align}
and find a condition where the eigenvalues are nonnegative. Here, positive definite matrix $\Gamma$ is diagonalized as $\Gamma = UDU^{-1}$ where $U$ is unitary and $D=\text{diag}(e_{1},\dots,e_{d})$ where the entries are the eigenvalues of $\Gamma$. Note that $e_{i}$ satisfies $s\leq e_{i}\leq s\sqrt{\frac{\beta}{\alpha}}$.
Then, the eigenvalues of the matrix $M_{1}+M_{2}-k_{1}E$ are equivalent to eigenvalues of a matrix
\begin{align}
G_{1}=
\begin{bmatrix}
2aD^{-1}-k_{1}\alpha^{-1}I-bk_{1}D^{-2} & -ak_{1}D^{-1}\\
-ak_{1}D^{-1} &2(c-as^{-2})D+ 2\alpha^{-1}D-k_{1}\alpha^{-1}I-ck_{1}I
\end{bmatrix}.
\end{align}
The eigenvalues of the matrix $G_{1}$ are union of the eigenvalues of a matrix $G_{1i}$, which is 
\begin{align}
G_{1i} = 
\begin{bmatrix}
2ae_{i}^{-1}-k_{1}\alpha^{-1}-bk_{1}e_{i}^{-2} & -ak_{1}e_{i}^{-1}\\
-ak_{1}e_{i}^{-1} &2(c-as^{-2})e_{i}+ 2e_{i}\alpha^{-1}-k_{1}\alpha^{-1}-ck_{1}
\end{bmatrix}.
\end{align}
Letting $a =\alpha^{-1}a_{1}, b = \alpha^{-1}b_{1}$, and $c = \alpha^{-1}c_{1}$, eigenvalues of the matrix $G_{1i}$ are positive scalar multiple of a matrix \begin{equation}
G '_{1i} = 
\begin{bmatrix}
2a_{1}e_{i}^{-1}-k_{1}-b_{1}k_{1}e_{i}^{-2} & -a_{1}k_{1}e_{i}^{-1}\\
-a_{1}k_{1}e_{i}^{-1} &2(c_{1}-a_{1}s^{-2})e_{i}+ 2e_{i}-k_{1}-c_{1}k_{1}
\end{bmatrix}.
\end{equation}
Here, for $G_{1i}'$, we denote $a$ for $a_{1}$, $b$ for $b_{1}$, $c$ for $c_{1}$, $k$ for $k_{1}$, and $d$ for $e_{i}^{-1}$ for convenience.  
Then
\begin{equation}
G'_{1i} = 
\begin{bmatrix}
2ad-k(1+bd^{2}) & -kad\\
-kad &2(1+c-as^{-2})d^{-1}-(1+c)k
\end{bmatrix}
=
\begin{bmatrix}
-pk+2q & -tk\\
-tk &-rk+2u
\end{bmatrix}
\end{equation}
where $p = 1+bd^{2}, q = ad, r = 1+c, t = ad$, and $u = (1+c-as^{-2})d^{-1}$ are positive.\par
Thirdly, letting the function $f(d)$ as
\begin{align}
f(d) = \frac{2qu}{pu+qr}(d),
\end{align} we prove that $G'_{1i}$ has two nonnegative eigenvalues when $k = \min_{d\in[s^{-1}\sqrt{\frac{\alpha}{\beta}},s^{-1}]}f(d)$. Note that since the denominator of $f$ is a quadratic function of $d$, $f$ has minimum at endpoints of an interval so $k = \min \{f(s^{-1}\sqrt{\frac{\alpha}{\beta}}),f(s^{-1})\}$ holds. Based on \cref{thm: cayley}, we prove that
\begin{align}
2(q+u)\geq (p+r)k
\end{align}
and
\begin{align}
prk^{2}-2(qr+pu)k+4qu\geq k^{2}t^{2}.
\end{align}
For the first inequality,
\begin{align}
\frac{2(q+u)}{p+r}\geq \frac{2qu}{pu+qr} \geq k
\end{align}
holds. In the case of second inequality, the equation 
\begin{align}
(pr-t^{2})x^{2}-2(qr+pu)x+4qu=0
\end{align}
has two positive solutions. Considering the smaller solution of this equation,
\begin{align}
    \frac{qr+pu-\sqrt{(qr+pu)^{2}-4qu(pr-t^{2})}}{pr-t^{2}} =& \frac{4qu}{qr+pu+\sqrt{(qr+pu)^{2}-4qu(pr-t^{2})}}\\\geq&\frac{2qu}{qr+pu}\geq k,
\end{align}
so the inequality holds since it is a quadratic function of $k$.\par
Fourth, we find ${k}_{2}$ that satisfies the $M_{3}\succeq -k_{2}E$.
The matrix $M_{3}+k_{2}E$ is given as
\begin{align}
G_{2} = M_{3}+k_{2}E = 
\begin{bmatrix}
-2L_{1}\gamma^{2}I+k_{2}\alpha^{-1}I+bk_{2}\Gamma^{-2}& ak_{2}\Gamma^{-1}\\
ak_{2}\Gamma^{-1}& -2L_{2}\gamma^{2}I+k_{2}\alpha^{-1}I+ck_{2}I
\end{bmatrix}
\end{align}
and similarly we consider the matrix
\begin{align}
G'_{2} = 
\begin{bmatrix}
-2\alpha\gamma^{2} L_{1}I+k_{2}I+b_{1}k_{2}\Gamma^{-2}& a_{1}k_{2}\Gamma^{-1}\\
a_{1}k_{2}\Gamma^{-1}& -2\alpha\gamma^{2} L_{2}I+k_{2}I+c_{1}k_{2}I
\end{bmatrix}
\end{align}
where $a =\alpha^{-1}a_{1}, b = \alpha^{-1}b_{1}$ and $c = \alpha^{-1}c_{1}$. Then, its eigenvalues are the union of the eigenvalues of 
\begin{align}
G'_{2i} = 
\begin{bmatrix}
-2\alpha\gamma^{2} L_{1}+k_{2}(1+b_{1}e_{i}^{-2}) & a_{1}e_{i}^{-1}k_{2}\\
a_{1}e_{i}^{-1}k_{2} & -2\alpha\gamma^{2} L_{2}+k_{2}(1+c_{1})
\end{bmatrix},
\end{align}
and we use the notation $a$ for $a_{1}$, $b$ for $b_{1}$, $c$ for $c_{1}$, $k$ for $k_{2}$, and $d$ for $e_{i}^{-1}$ for convenience.  
Then, $G'_{2i}$ is given as 
\begin{align}
G'_{2i} = 
\begin{bmatrix}
-2\alpha\gamma^{2} L_{1}+k(1+bd^{2}) & adk\\
adk & -2\alpha\gamma^{2} L_{2}+k(1+c)
\end{bmatrix}=
\begin{bmatrix}
pk-2q & kt\\kt & rk-2u
\end{bmatrix}
\end{align}
where $p = 1+bd^{2}, q = \alpha\gamma^{2} L_{1}, r = 1+c, t = ad,$ and $u = \alpha\gamma^{2} L_{2}$ are positive coefficients. Then, letting the function $g$ as
\begin{align}
g(d) = \frac{2(pu+qr)}{(pr-t^{2})\gamma^{2}}(d),
\end{align}
we prove $G'_{2i}$ has two nonnegative eigenvalues when $k =\gamma^{2} \max_{d\in[s^{-1}\sqrt{\frac{ \alpha}{\beta}},s^{-1}]}g(d)$. Since $g$ is a rational function of $d^{2}$, $k=\gamma^{2} \max \{g(s^{-1}\sqrt{\frac{\alpha}{\beta}}),g(s^{-1})\}$ holds. Similar to the previous derivation, $k$ has to satisfy
\begin{align}
    pk-2q+rk-2u\geq 0,
\end{align}
and it is proved since
\begin{align}
k \geq \frac{2(pu+qr)}{pr-t^{2}}\geq \frac{2(q+u)}{p+r}.
\end{align}
In addition, $k$ has to satisfy
\begin{align}
(pk-2q)(rk-2u)-t^{2}k^{2} = (pr-t^{2})k^{2}-2(pu+qr)k+4qu\geq 0,
\end{align}
and it is proved since $k$ is larger than the larger solution of the equation
\begin{align}
(pr-t^{2})x^{2}-2(pu+qr)x+4qu=0,
\end{align}
which is 
\begin{align}
    k \geq \frac{2(pu+qr)}{pr-t^{2}}\geq \frac{pu+qr+\sqrt{(pu+qr)^{2}-4qu(pr-t^{2})}}{pr-t^{2}}.
\end{align}
\par Therefore, we proved that there exists constant $k_{1}, k_{2}>0$ that satisfies $M_{1}+M_{2}\succeq k_{1}E$ and $M_{3}\succeq- k_{2}E$. It leads that $M_{1}+M_{2}+M_{3}\succeq (k_{1}-k_{2})E$, and since $k_{1} = m_{1}$ and $k_{2} = \gamma^{2}m_{2}$ the assumption on $\gamma$ gives that $k_{1}>k_{2}$. Letting the cross term of a Lyapunov function as $\mathcal{L}_{\text{cross}}(\rho_{t}) = \mathbb{E}_{\pi}\big[\,\big\langle\, \nabla_{\mathbf{x}}\frac{\rho_{t}}{\pi},\,S\,\nabla_{\mathbf{x}}\frac{\rho_{t}}{\pi}\,\big\rangle\,\big]$, by Poincar\'e's inequality, 
\begin{align}
\frac{\dd \mathcal{L}(\rho_{t})}{\dd t} \leq& -\mathbb{E}_{\pi}\big[\,\big\langle\, \nabla_{\mathbf{x}}\frac{\rho_{t}}{\pi},\,(M_{1}+M_{2}+M_{3})\,\nabla_{\mathbf{x}}\frac{\rho_{t}}{\pi}\,\big\rangle\,\big]\\ \leq& -\mathbb{E}_{\pi}\big[\,\big\langle\, \nabla_{\mathbf{x}}\frac{\rho_{t}}{\pi},\,(k_{1}-k_{2})(\frac{1}{\alpha}I+S)\,\nabla_{\mathbf{x}}\frac{\rho_{t}}{\pi}\,\big\rangle\,\big]\\=&-(m_{1}-\gamma^{2}m_{2})\big[\frac{1}{\alpha}\mathbb{E}\big[\Vert\nabla_{\mathbf{x}}\frac{\rho_{t}}{\pi}\Vert^{2}\big]+\mathcal{L}_{\text{cross}}(\rho_{t})\big]\\\leq&-(m_{1}-\gamma^{2}m_{2})\big[\mathbb{E}_{\pi}\big[(\frac{\rho_{t}}{\pi}-1)^{2}\big]+\mathcal{L}_{\text{cross}}(\rho_{t})\big]\\=&-(m_{1}-\gamma^{2}m_{2})\mathcal{L}(\rho_{t}) 
\end{align}
holds. Then, by Gronwall's inequality, we completes the proof.
\end{proof}
We note that for any constant $s$, there exist coefficients $a, b$, and $c$ that satisfies every assumptions of the theorem, since the positive definite condition of $S$ is equal to $bc-a^{2}>0$. Then, we optimize the constant $s$ and coefficients $a, b$, and $c$ to improve the convergence rate. In our derivation, we obtain the optimal convergence rate when $s = 2$, which coincides with our intuition.

\begin{theorem}\label{cor17}
Recalling the notations of \cref{thm13}, $m_{1}$ is highest when $s = 2$. In particular, $s = 2$, $a = \alpha^{-1}(2x+2)$, $b = \alpha^{-1}(12x+8)$, and $c = \alpha^{-1}x$ satisfies the assumption of \cref{thm13}, and for any $\epsilon>0$, there exists sufficiently large $x>\frac{1}{\sqrt{2}}$ such that $m_{1} = \frac{1}{2}$ holds. Here, $m_{2} = O(\beta^{2}\alpha^{-3}d)$ holds for the worst case, so $m_{2}$ is sufficiently small when $\alpha^{3}\gg \beta^{2}d$.
\end{theorem}
\begin{proof}
We recall the notations of \cref{thm13} including the function $f$. Define a function $d(k) = ks^{-1}$ on a domain $k\in\{1,\sqrt{\frac{\alpha}{\beta}}\}$, and denote $a_{1}= \alpha a$, $b_{1} = \alpha b$, and $c_{1} = \alpha c$. Then, by the assumption $b_{1} = (a_{1}+c_{1})s^{2}$ and letting $w = 1+c_{1} > a_{1}s^{-2}$,  
\begin{align}
f(d)=&\frac{2a_{1}d(1+c_{1}-a_{1}s^{-2})}{(1+b_{1}d^{2})(1+c_{1}-a_{1}s^{-2})+(1+c_{1})a_{1}d^{2}}\\=&\frac{2a_{1}s^{-1}(1+c_{1}-a_{1}s^{-2})}{(k^{-1}+(a_{1}+c_{1})k)(1+c_{1}-a_{1}s^{-2})+a_{1}(1+c_{1})ks^{-2}}\\=&\frac{2a_{1}s^{-1}(w-a_{1}s^{-2})}{(k^{-1}-k+a_{1}k+k w)(w-a_{1}s^{-2})+a_{1}ks^{-2}w}\\\leq& \frac{2a_{1}s^{-1}(w-a_{1}s^{-2})}{(a_{1}k+k w)(w- a_{1}s^{-2})+a_{1}ks^{-2}w}
\end{align}
holds. Finally, letting $t = \frac{a_{1}}{w}<s^{2}$,
\begin{align}
f(d)\leq &\frac{2t(1-s^{-2}t)s^{-1}}{(kt+k)(1-s^{-2}t)+tks^{-2}} =\frac{2s^{-1}}{k} \frac{t-s^{-2}t^{2}}{1+t-s^{-2}t^{2}}=\frac{2s^{-1}}{k}\frac{\frac{s^{4}}{4}-(\frac{s^{2}}{2}-t)^{2}}{s^{2}+\frac{s^{4}}{4}-(\frac{s^{2}}{2}-t)^{2}}\\ 
\leq & \frac{2s^{-1}}{k}\frac{\frac{s^{4}}{4}}{s^{2}+\frac{s^{4}}{4}} = \frac{2s^{-1}}{k(4s^{-2}+1)}\leq \frac{1}{2k}
\end{align}
holds, and equality holds when $k = 1$, $t = \frac{s^{2}}{2}$, and $s = 2$. When $s = 2$, by the definition of $m_{1}$, $m_{1}= \min \{f(s^{-1}\sqrt{\frac{\alpha}{\beta}}),f(s^{-1})\}\leq \frac{1}{2}$, and when $s\neq 2$, $m_{1}\leq\frac{2s^{-1}}{4s^{-2}+1}<\frac{1}{2}$ holds. \par
Now, we consider the case where $s = 2$, $a = \alpha^{-1}(2x+2)$, $b = \alpha^{-1}(12x+8)$, and $c = \alpha^{-1}x$ for some $x>\frac{1}{\sqrt{2}}$. Since
\begin{align}
a + c = \alpha^{-1}(3x+2) = s^{-2}b\,\, \text{and}\,\,    \alpha^{-1}+c = \alpha^{-1}(x+1) > s^{-2}a
\end{align}
holds, $a, b, c$, and $s$ satisfies the assumption of \cref{thm13}. Then, $f(s^{-1})$ is given as 
\begin{align}
f(\frac{1}{2}) =\frac{\frac{1}{2}(2x+2)(x+1)}{\frac{1}{2}(3x+3)(x+1)+\frac{1}{2}(x+1)(x+1)}  = \frac{1}{2}, 
\end{align}
and letting $\sqrt{\frac{\alpha}{\beta}} = k<1$,
\begin{align}
f(\frac{1}{2}k) = \frac{k(x+1)^{2}}{\frac{1}{2}(1+(3x+2)k^{2})(x+1)+\frac{1}{2}k^{2}(x+1)^{2}}=\frac{1}{2k+\frac{1-k^{2}}{2k^{2}(x+1)}}\geq \frac{1}{2}
\end{align}
holds for sufficiently large $x$, which is $O(\beta\alpha^{-1})$ for the worst case.
Then, since
\begin{align}
g(s^{-1}\sqrt{\frac{\alpha}{\beta}} )=\frac{O(x^{2}\beta\alpha^{-4}d)}{O(x\alpha^{-2})+O(x^{2}\alpha^{-1}\beta^{-1})} = O(\beta^{2}\alpha^{-3}d)    
\end{align}
and
\begin{align}
g(s^{-1}) = \frac{O(x^{2}\beta\alpha^{-4}d)}{O(x^{2}\alpha^{-2})}  = O(\beta\alpha^{-2}d)
\end{align} 
holds, $m_{1} = \frac{1}{2}$ and $m_{2} = O(\beta^{2}\alpha^{-3}d)$ so the proof is complete.
\end{proof}

\begin{corollary}\label{thm20}
Consider the original dynamics \cref{eq: original} where the friction coefficient is given by $\Gamma = 2\sqrt{\text{Hess}\, \, V}$. Then,
\begin{align}
\chi^{2}(\rho_{t}\Vert \pi) \leq C\, e^{-\sqrt{\alpha}(\frac{1}{2}-\gamma^{2}m_{2})t}
\end{align}
for some constant $C$, and $\gamma^{2}m_{2}$ is sufficiently small when $\gamma^{2}\beta^{2}d \ll \alpha^{3}$. 
\end{corollary}

Based on this approach, we evaluate the convergence rate of the original dynamics \cref{eq: original} with diagonal quadratic potential. We prove that our approach derives the convergence rate very close to the optimal convergence rate discussed in \cref{sec42}.
\begin{theorem}
Consider the rescaled dynamics \cref{eq: rescaled} with diagonal quadratic constant potential $V(\mathbf{q}) = \sum_{i = 1}^{d}(v_{i}^{2}/2)\mathbf{q}_{i}^{2}$ and friction coefficient $\Gamma = 2\sqrt{\text{Hess}\, \, V}$. Then, for any $\epsilon>0$, there exists $C(\epsilon)$ that satisfies
\begin{align}
\mathcal{L}(\rho_{t}) \leq e^{-(2-\epsilon)t}\mathcal{L}(\rho_{0}).        
\end{align}
Therefore, we have a convergence rate $2-\epsilon$ for any small $\epsilon$.
\end{theorem}
\begin{proof}
Recall $M_{1}$, $M_{2}$, and $E$ from the proof of \cref{thm13} and prove that there exists $a$, $b$, and $c$ that satisfies $M_{1}+M_{2}-(2-\epsilon)E \succeq 0$. Denote $k = \frac{\epsilon}{2}$, and let $b = 2(a+x)$ and $c = \frac{1}{2}(a-y)$ for convenience. Then, by the similar derivation with \cref{thm13}, we only need to prove that for sufficient $a$, $x$, and $y$,
\begin{align}
G_{1i} = 
\begin{bmatrix}
g_{1}(v_{i}, a) & g_{2}(v_{i}, a)\\ 
g_{2}(v_{i}, a)& g_{3}(v_{i}, a)
\end{bmatrix}
\succeq 0
\end{align}
with
$g_{1}(v_{i}, a) = (v_{i}^{-1} -(1-k)v_{i}^{-2})a- (1-k)(x v_{i}^{-2}+2\alpha^{-1})$, $g_{2}(v_{i}, a) = (1-(1-k)v_{i}^{-1})a - \frac{1}{2}(x+y)$, and $g_{3}(v_{i}, a) = (v_{i}-(1-k))a-(2v_{i}-(1-k))(y-2\alpha^{-1})$ holds for all $i$. 

Now, we let $0<y<x$ and $a$ sufficiently large. Then, $a > 0$, $b = 2(a+x)>0$, $c = \frac{1}{2}(a-y)>0$, and $bc -a^{2} = (x-y)a - xy > 0$ holds, so the assumptions on Lyapunov function holds. In addition, since $v_{i}\geq 1$, $g_{1}(v_{i}, a)$ and $g_{3}(v_{i}, a)$ are positive for sufficiently large $a$. Finally, 
\begin{align}
&g_{1}(v_{i}, a)g_{3}(v_{i}, a)-g_{2}(v_{i}, a)^{2} \\
&= (1-(1-k)v_{i}^{-1})a(x+y-(1-k)(xv_{i}^{-1}+2\alpha^{-1}v_{i})-(2-(1-k)v_{i}^{-1})(y-2\alpha^{-1}))\\
&+(1-k)(x v_{i}^{-2}+2\alpha^{-1})(2v_{i}-(1-k))(y-2\alpha^{-1})-\frac{1}{4}(x+y)^{2}
\end{align}
and
$1-(1-k)v_{i}^{-1}>0$ and 
\begin{align}
    x&+y-(1-k)(xv_{i}^{-1}+2\alpha^{-1}v_{i})-(2-(1-k)v_{i}^{-1})(y-2\alpha^{-1})\\
    &= (1-(1-k)v_{i}^{-1})(x-y)+2\alpha^{-1}(2-(1-k)(v_{i}+v_{i}^{-1})) 
\end{align}
so we could set sufficiently large $a$ and $0<y<x$ that satisfies $g_{1}(v_{i}, a)g_{3}(v_{i}, a) - g_{2}(v_{i}, a)^{2}>0$ for all $i$. It leads that $M_{1}+M_{2}\succeq (2-\epsilon) E$. Therefore, by applying the Poincar\'e's inequality as in \cref{thm13}, the proof is complete. 
\end{proof}
\begin{corollary}\label{thm21}
For the original dynamics \cref{eq: original} with diagonal quadratic potential $V(\mathbf{q}) = \sum_{i = 1}^{d}(v_{i}^{2}/2)\mathbf{q}_{i}^{2}$ and friction coefficient $\Gamma = 2\sqrt{Hess V}$, for any $\epsilon>0$
\begin{align}
\chi^{2}(\rho_{t}\Vert \pi) \leq C(\epsilon)\, e^{-\sqrt{\alpha}(2-\epsilon)t}
\end{align}
holds for some constant $C(\epsilon)$ which is independent of $t$.
\end{corollary}

\section{Comparison with constant scalar friction case}\label{app1}
In this section, we compare the convergence rate with constant scalar friction coefficient case based on the result of \cref{sec3} and \cite{8}. We give explicit convergence rate comparison of the dynamics with specific potential with diagonal Hessian for the special case. 

\subsection{Comparison with DMS hypocoercive estimation result}

Here, convergence rates of original dynamics \cref{eq: original} with different friction coefficients are compared based on the results of \cref{sec3} and \cite{8}. We address \cref{thm16} for the constant scalar friction coefficient case, which is a reformulation of \cite[Theorem 4]{8} and \cite[Theorem B.1]{8} for our assumptions. We refer the discussions in \cite[Remark 1.2]{8} and note that this approach is based on \cite[Theorem 1]{32}. 

\begin{theorem}\label{thm16}(See \cite[Theorem 4]{8}) Consider the case where the friction coefficient is $\Gamma = \lambda I$. Define the function $\lambda_{DMS} = \lambda_{DMS}(\lambda, \alpha, \epsilon)$ as 
\begin{align}
    \lambda_{DMS}(\lambda, \alpha, \epsilon) =\frac{\lambda-\frac{\epsilon}{1+\alpha}-\sqrt{\epsilon^{2}(\sqrt{2}+\frac{\lambda}{2})^{2}+(\lambda-\frac{2\alpha+1}{\alpha+1}\epsilon)^{2}}}{2(1+\vert \epsilon \vert)}.
\end{align}

Then, consider the solution $f(t, \mathbf{q}, \mathbf{p})$ of backward Kolmogorov equation of the original dynamics \cref{eq: original} that satisfies $\int f(0, \mathbf{q}, \mathbf{p})\, \dd \pi = 0$. For $\epsilon \in (-1, 1)$ that satisfies $\lambda_{DMS}(\lambda, \alpha, \epsilon)>0$,
\begin{align}
    \Vert f(t, \mathbf{q}, \mathbf{p})\Vert_{L^{2}(\pi)}\, \leq\, \sqrt{\frac{1+\vert \epsilon \vert}{1-\vert \epsilon \vert}}\,\Vert f(0, \mathbf{q}, \mathbf{p})\Vert_{L^{2}(\pi)}\,e^{-\lambda_{DMS} t}
\end{align}
holds. In addition, by the direct consequence of the above inequality and \cite[Remark 1.2]{8}, 
\begin{align}
    \chi^{2}(\rho_{t}\Vert \pi)\, \leq\, \frac{1+\vert \epsilon \vert}{1-\vert \epsilon \vert}\,\chi^{2}(\rho_{0}\Vert \pi)\,e^{-2\lambda_{DMS} t}
\end{align}
holds.
\end{theorem}
From the above theorem, we have a convergence rate result for constant scalar friction coefficient case. Now, we compare this convergence rate with our result.  
\begin{theorem}\label{thm17}
Recall the convergence rate $2\lambda_{DMS}$ for constant scalar friction coefficient case in \cref{thm16} and denote $\Lambda_{DMS}(\lambda, \alpha) = \sup_{\epsilon\in(-1, 1)}\lambda_{DMS}(\lambda, \alpha, \epsilon)$. In addition, recall the convergence rate $\frac{\sqrt{\alpha}}{2} - O(\alpha^{-\frac{1}{2}}\kappa^{2}\gamma^{2}d)$ for our potential function related friction coefficient case. Then, considering the case of $\beta^{2}\gamma^{2}d\ll \alpha^{3}$, $2\Lambda_{DMS}$ is smaller than $\frac{\sqrt{\alpha}}{2} - O(\alpha^{-\frac{1}{2}}\kappa^{2}\gamma^{2}d)$.   
\end{theorem}
\begin{proof}
Denote the set $A\subset (-1, 1)$ as $A =  \{\epsilon\,\vert\, \epsilon \in (-1, 1),\,  \lambda_{DMS} > 0\}$
and note that $A$ is not an empty set and every element of $A$ is positive. Then, 
\begin{align}
    \Lambda_{DMS} &=\sup_{\epsilon\in A}\, \frac{\lambda-\frac{\epsilon}{1+\alpha}-\sqrt{\epsilon^{2}(\sqrt{2}+\frac{\lambda}{2})^{2}+(\lambda-\frac{2\alpha+1}{\alpha+1}\epsilon)^{2}}}{2(1+\epsilon)}\\
    & = \sup_{\epsilon\in A}\, \frac{(\lambda-\frac{\epsilon}{1+\alpha})^{2}-\epsilon^{2}(\sqrt{2}+\frac{\lambda}{2})^{2}-(\lambda-\frac{2\alpha+1}{\alpha+1}\epsilon)^{2}}{2(1+ \epsilon)(\lambda-\frac{\epsilon}{1+\alpha}+\sqrt{\epsilon^{2}(\sqrt{2}+\frac{\lambda}{2})^{2}+(\lambda-\frac{2\alpha+1}{\alpha+1}\epsilon)^{2}})}\\
    &\leq \sup_{\epsilon\in A}\, \frac{(\frac{4\alpha\epsilon}{\alpha+1}-\sqrt{2}\epsilon^{2})\lambda}{2(1+\epsilon)(\lambda-\frac{\epsilon}{1+\alpha}+\epsilon(\sqrt{2}+\frac{\lambda}{2}))}\\
    &\leq \sup_{\epsilon\in A}\, \frac{\frac{4\alpha\epsilon}{\alpha+1}-\sqrt{2}\epsilon^{2}}{2(1+\epsilon)(1+\frac{\epsilon}{2})} = \sup_{\epsilon\in A}\, (\frac{4\alpha}{\alpha + 1}-\sqrt{2}\epsilon) \frac{\epsilon}{\epsilon^{2}+3\epsilon+2}
\end{align}
holds. 

Considering our convergence rate $\frac{\sqrt{\alpha}}{2} - O(\alpha^{-\frac{1}{2}}\kappa^{2}\gamma^{2}d) = \frac{\sqrt{\alpha}}{2}(1 - O(\alpha^{-1}\kappa^{2}\gamma^{2}d))$, it suffices to prove that $2\Lambda_{DMS}\leq \frac{\sqrt{\alpha}}{2.01}$ by the assumption. Therefore, we prove that
\begin{align}
\sup_{\epsilon\in A}\, (\frac{4\alpha}{\alpha + 1}-\sqrt{2}\epsilon) \frac{\epsilon}{\epsilon^{2}+3\epsilon+2}\leq \frac{\sqrt{\alpha}}{4.02}. 
\end{align}
Since
\begin{align}
    (\frac{\epsilon^{2}+3\epsilon+2}{4.02\epsilon}&\sqrt{\alpha}+\sqrt{2}\epsilon)(\alpha+1)\\&\geq \frac{\epsilon^{2}+3\epsilon+2}{4.02\epsilon}(\alpha\sqrt{\alpha}+\sqrt{\alpha})+\sqrt{2}\epsilon\alpha\\&\geq (\frac{\epsilon^{2}+3\epsilon+2}{2.01\epsilon} + \sqrt{2}\epsilon)\alpha=((\frac{1}{2.01}+\sqrt{2})\epsilon+\frac{3}{2.01}+\frac{2}{2.01\epsilon})\alpha \\ &\geq (2\sqrt{(\frac{1}{2.01}+\sqrt{2})(\frac{2}{2.01})}+\frac{3}{2.01})\alpha\geq 4.2 \alpha >4 \alpha, 
\end{align}
the proof is complete. 
\end{proof}

Therefore, we proved that our convergence rate result is better than the rate of any constant scalar friction coefficient case obtained in \cite[Theorem 4]{8} under sufficient condition.

\subsection{Case of potential function with diagonal Hessian}\label{sec42}
We start with the case of original dynamics \cref{eq: original} with diagonal quadratic potential, which is $V(\mathbf{q})= \sum_{i=1}^{d}(v_{i}^{2}/2)\mathbf{q}_{i}^{2}$ for $v_{i}>0$. We prove that the dynamics with friction coefficient $\Gamma = 2\sqrt{\text{Hess}\, \, V}$ has grater or equal convergence rate than the dynamics with friction coefficient $\Gamma = \lambda I$ for every $\lambda> 0$, in case of $\chi^{2}$-divergence. Computing the bound of $\chi^{2}$-divergence based on the explicit solution is not very straightforward, so we contained it in the proof. 

\begin{theorem}\label{theorem19} Consider the original dynamics \cref{eq: original} with the potential $V(\mathbf{q})= \sum_{i=1}^{d}(v_{i}^{2}/2)\mathbf{q}_{i}^{2}$ for all $v_{i}>0$. Then, for every $\lambda > 0$, exponential convergence rate of $\chi^{2}$-divergence is greater or equal in friction coefficient $\Gamma = 2\sqrt{\text{Hess}\, \, V}$ case than $\Gamma = \lambda I$ case.
\end{theorem}
\begin{proof}
Before analyzing the general process, we consider the equation where $\mathbf{q}_{t}$ and $\mathbf{p}_{t}$ are one dimensional variables. For $\lambda, w>0$, we prove that the exponential convergence rate of two-dimensional SDE
\begin{equation}\label{eq7}
\begin{cases}
\dd \mathbf{q}_{t} = \mathbf{p}_{t}\, \dd t \\
\dd\mathbf{p}_{t} = -w^{2}\mathbf{q}_{t}\,\dd t-\lambda\mathbf{p}_{t}\,\dd t + \sqrt{2\lambda}\,\dd\mathbf{W}_{t}
\end{cases}
\end{equation} 
is $\lambda-\sqrt{\lambda^{2}-4w^{2}}$ when $\lambda > 2w$ and
$\lambda$ when $\lambda \leq 2w$. 
Note that the maximum convergence rate is $2w$ when $\lambda = 2w$. The equation \cref{eq7} is equivalent to the form
\begin{align}
\dd \mathbf{X}_{t} = AX_{t}\dd t + \sigma \dd \mathbf{W}_{t}    
\end{align}
where $X_{t}=\bigl[\begin{smallmatrix}\mathbf{q}_{t}\\ \mathbf{p}_{t}\end{smallmatrix}\bigr]$, $A =\bigl[ \begin{smallmatrix}0& 1 \\-w^{2}&-\lambda \end{smallmatrix}\bigr]$, and $\sigma =\bigl[ \begin{smallmatrix}0&0 \\ 0& \sqrt{2 \lambda}\end{smallmatrix}\bigr]$. The solution of this equation is
\begin{align}
    X_{t} = e^{At}X_{0}+M_{t}
\end{align}
where $M_{t}\sim\mathcal{N}(0,\Sigma)$ and
\begin{align}
    \Sigma =\int_{0}^{t}e^{A(t-s)}\sigma\sigma^{\top}e^{A^{\top}(t-s)}ds.
\end{align}\par
First, we consider the case where $\lambda = 2w$. Since $A$ has one eigenvalue $-w$ and $(A+wI)^{2}=0$ holds,
\begin{align}
e^{At}= e^{-wtI}(I+(A+wI)t) =e^{-wt}\begin{bmatrix}1+wt & t\\ -w^{2}t& 1-wt
\end{bmatrix}
\end{align}
and similarly
\begin{align}
e^{A^{\top}t}= e^{-wtI}(I+(A^{\top}+wI)t) =e^{-wt}\begin{bmatrix}1+wt & -w^{2}t\\ t& 1-wt
\end{bmatrix}
\end{align}
holds. Then, by direct computation, the mean of $X_{t}$ is $\mu_{1} = e^{-wt} N_{1}$, covariance matrix is $\Sigma_{1} =\bigl[ \begin{smallmatrix} \frac{1}{w^{2}}& 0 \\ 0 & 1 \end{smallmatrix}\bigr] +e^{-2wt}N_{2}$, and inverse of covariance matrix is 
$\Sigma_{1}^{-1}= \bigl[\begin{smallmatrix} w^{2}&0 \\ 0& 1 \end{smallmatrix}\bigr] +e^{-2wt}N_{3},$
where $N_{1}$, $N_{2}$, and $N_{3}$ is a matrix which elements are bounded as $O(t^{k_{1}})$ for some finite $k_{1}$. Since its stationary distribution $\pi$ is a Gaussian with mean $\mu_{2}=0$ and covariance $\Sigma_{2}=\bigl[\begin{smallmatrix}
\frac{1}{w^{2}}& 0 \\ 0 & 1
\end{smallmatrix}\bigr]$, we obtain
\begin{align}
    \chi^{2}(&\rho_{t}\Vert \pi)+1\\=& \int\frac{\sqrt{\det \Sigma_{2}}}{2\pi \det \Sigma_{1}}e^{-(x-\mu_{1})^{\top}\Sigma_{1}^{-1}(x-\mu_{1})+\frac{1}{2}(x-\mu_{2})^{\top}\Sigma_{2}^{-1}(x-\mu_{2})}\dd x\\=&\int\frac{\sqrt{\det \Sigma_{2}}}{2\pi \det \Sigma_{1}}e^{-(x-(\Sigma_{1}^{-1}-\frac{1}{2}\Sigma_{2}^{-1})^{-1}\Sigma_{1}^{-1}\mu_{1}))^{\top}(\Sigma_{1}^{-1}-\frac{1}{2}\Sigma_{2}^{-1})(x-(\Sigma_{1}^{-1}-\frac{1}{2}\Sigma_{2}^{-1})^{-1}\Sigma_{1}^{-1}\mu_{1})}\\&\cdot e^{((\Sigma_{1}^{-1}-\frac{1}{2}\Sigma_{2}^{-1})^{-1}\Sigma_{1}^{-1}\mu_{1})^{\top}(\Sigma_{1}^{-1}-\frac{1}{2}\Sigma_{2}^{-1})((\Sigma_{1}^{-1}-\frac{1}{2}\Sigma_{2}^{-1})^{-1}\Sigma_{1}^{-1}\mu_{1})}\dd x\\=&1+e^{-2wt}\cdot f(t),
\end{align}
where $f(t)$ is a function bounded as $O(t^{k_{2}})$ for some finite $k_{2}$. Therefore, $\chi^{2}$-divergence converges exponentially with rate $2w$.\par
Second, we consider the case where $\lambda \neq 2w$. We let two distinct eigenvalues as $\alpha_{1} = \frac{-\lambda+\sqrt{\lambda^{2}-4w^{2}}}{2}$ and $\alpha_{2} = \frac{-\lambda-\sqrt{\lambda^{2}-4w^{2}}}{2}$. Then, similarly evaluating the mean and the covariance matrix, we obtain the mean and the covariance of $X_{t}$ as
\begin{align}
\mu_{1}=\frac{1}{\alpha_{2}-\alpha_{1}}
\begin{bmatrix}
\alpha_{2}e^{\alpha_{1}t}-\alpha_{1}e^{\alpha_{2}t} & e^{\alpha_{2}t}-e^{\alpha_{1}t}\\ \alpha_{1}\alpha_{2}(e^{\alpha_{1}t}-e^{\alpha_{2}t})&\alpha_{2}e^{\alpha_{2}t}-\alpha_{1}e^{\alpha_{1}t}   
\end{bmatrix} X_{0}
\end{align}
and
\begin{align}
\Sigma_{1} = \frac{2\lambda}{(\alpha_{1}-\alpha_{2})^{2}}
\begin{bmatrix}
p+2q+r & \alpha_{1}(p+q)+\alpha_{2}(q+r)\\
\alpha_{1}(p+q)+\alpha_{2}(q+r) & \alpha_{1}^{2}p+2\alpha_{1}\alpha_{2}q+\alpha_{2}^{2}r
\end{bmatrix}
\end{align}
where $p = \frac{1}{2\alpha_{1}}(e^{2\alpha_{1}t}-1)$, $q =- \frac{1}{\alpha_{1}+\alpha_{2}}(e^{(\alpha_{1}+\alpha_{2})t}-1)$, and $r = \frac{1}{2\alpha_{2}}(e^{2\alpha_{2}t}-1)$.\par
When $\lambda>2w$, similar to the previous derivation, we obtain $\mu_{1} = e^{\alpha_{1}t} N_{1}$, $\Sigma_{1} = \bigl[\begin{smallmatrix} \frac{1}{w^{2}} & 0\\ 0 & 1
\end{smallmatrix}\bigr]
+e^{2\alpha_{1}t} N_{2}$,
and $\Sigma_{1}^{-1} = \bigl[\begin{smallmatrix}w^{2} & 0\\ 0 & 1\end{smallmatrix}\bigr]+e^{2\alpha_{1}t} N_{3},$
where $N_{1}$, $N_{2}$ and $N_{3}$ is a matrix where the elements are given as $O(t^{k_{3}})$ for some $k_{3}$. Then, $\chi^{2}$-divergence converges exponentially with rate $-2\alpha_{1}=\lambda-\sqrt{\lambda^{2}-4w^{2}}$.\par
When $\lambda<2w$, since the real part of $\alpha_{1}$ and $\alpha_{2}$ is $-\frac{\lambda}{2}$, we can similarly let $\mu_{1} = e^{-\frac{\lambda}{2}t} N_{1}$, $\Sigma_{1} = \bigl[\begin{smallmatrix} \frac{1}{w^{2}} & 0\\ 0 & 1 \end{smallmatrix}\bigr] +e^{-\lambda t} N_{2}$,
and 
$\Sigma_{1}^{-1} = \bigl[\begin{smallmatrix}
w^{2} & 0\\ 0 & 1
\end{smallmatrix}\bigr]
+e^{-\lambda t} N_{3},$
where $N_{1}$, $N_{2}$, and $N_{3}$ is a matrix which elements are given as $O(t^{k_{4}})$ for some $k_{4}$. Then, $\chi^{2}$-divergence converges exponentially with rate $\lambda$.\par
Finally, we consider the original equation 
\begin{equation}
\begin{cases}
\dd \mathbf{q}_{t} = \mathbf{p}_{t}\, \dd t \\
\dd\mathbf{p}_{t} = -A\mathbf{q}_{t}\,\dd t-\Gamma\mathbf{p}_{t}\,\dd t + \sqrt{2\Gamma}\,\dd\mathbf{W}_{t}
\end{cases}
\end{equation}
where $A = \text{diag}(v_{1}^{2},\dots,v_{n}^{2})$ and assume $v_{1}\geq v_{2}\dots\geq v_{n}>0$ without the loss of generalization. Then, for both cases $\Gamma = \lambda I$ and $\Gamma = 2\sqrt{A}$, since we can split the equation to $n$ two-dimensional equations with the variables $(\mathbf{q}_{i},\mathbf{p}_{i})$, convergence rate of $\chi^{2}$-divergence follows the minimum value of the convergence rates of each two-dimensional equations. Therefore, when $\Gamma = 2\sqrt{A}$ we obtain the convergence rate $2 v_{n}$, and when $\Gamma = \lambda I$, the convergence rate is less or equal to $2 v_{n}$ and the equality is obtained when $\lambda = 2 v_{n}$. 
\end{proof}
From this theorem, we obtain that $\Gamma = 2\sqrt{\text{Hess}\, \, V}$ case has greater or equal convergence than $\Gamma = \lambda I$ case and intuitively understand that the condition $\Gamma = 2\sqrt{\text{Hess}\, \, V}$ is pertinent to the convergence of each variables. We present the final theorem which connects the diagonal quadratic potential case to general diagonal potential case.
\begin{theorem}\label{cor_comparison}

    Consider the potential function $V(\mathbf{q})$
    \begin{align}
        V(\mathbf{q}) = \frac{1}{2}\sum_{i = 1}^{d}v_{i}^{2}\mathbf{q}_{i}^{2}+\varepsilon\sum_{i = 1}^{d}f_{i}(\mathbf{q}_{i})
    \end{align}
    where $\varepsilon>0$ and $f_{i}^{(2)}(\mathbf{q}_{i})$ and $f_{i}^{(3)}(\mathbf{q}_{i})$ are bounded for all $i$. Then, there exists $\varepsilon_{0}>0$ such that for all $0<\varepsilon<\varepsilon_{0}$ the convergence rate $\frac{\sqrt{\alpha(\varepsilon)}}{2} - O(\alpha(\varepsilon)^{-\frac{1}{2}}\kappa(\varepsilon)^{2}\gamma(\varepsilon)^{2}d)$ is higher than $2\Lambda_{DMS}(\lambda,\alpha(\varepsilon))$.
\end{theorem}
\begin{proof}
    Let $\alpha = \min (v_{i}^{2})$ and $\beta = \max(v_{i}^{2})$. Then, by direct computation, 
    \begin{align}
        \text{Hess}\, V(\mathbf{q}) = \text{diag}(v_{1}^{2}+\varepsilon f_{1}^{(2)}(\mathbf{q}_{1}),\dots,v_{d}^{2}+\varepsilon f_{d}^{(2)}(\mathbf{q}_{d}))
    \end{align}
     holds, so $\alpha(\varepsilon)\preceq \text{Hess}\, V(\mathbf{q}) \preceq \beta(\varepsilon)$ where $\alpha(\varepsilon) \to \alpha$ and $\beta(\varepsilon) \to \beta$ as $\varepsilon \to 0$. In addition, 
    \begin{align}
        \frac{\dd (v_{i}^{2}+\varepsilon f_{i}^{(2)}(\mathbf{q}_{i}))^{\frac{1}{2}}}{\dd \mathbf{q}_{i}} = \frac{\varepsilon f_{i}^{(3)}(\mathbf{q}_{i})}{2(v_{i}^{2}+\varepsilon f_{i}^{(2)}(\mathbf{q}_{i}))^{\frac{1}{2}}} = O(\varepsilon),
    \end{align}     
    so $\gamma = \gamma(\varepsilon) = O(\varepsilon)$ holds. It leads that $\frac{\sqrt{\alpha(\varepsilon)}}{2} - O(\alpha(\varepsilon)^{-\frac{1}{2}}\kappa(\varepsilon)^{2}\gamma(\varepsilon)^{2}d)\to \frac{\sqrt{\alpha}}{2}$ as $\varepsilon \to 0$.
    
    Therefore, since $\Lambda_{DMS}(\lambda, \alpha)$ is a continuous function of $\alpha$, $\Lambda_{DMS}(\lambda, \alpha(\varepsilon))\to \Lambda_{DMS}(\lambda, \alpha)$ holds as $\varepsilon\to 0$. Combining the above results and applying the result of \cref{thm17} that $\frac{\sqrt{\alpha}}{2}>2\Lambda_{DMS}(\lambda, \alpha)$, the proof is complete.  
\end{proof}

\printbibliography 

\end{document}